\documentclass[10pt]{article}
\usepackage[fleqn]{amsmath}
\usepackage[cp1251]{inputenc}
\usepackage[english]{babel}
\usepackage{amsmath,amsfonts,amssymb,graphicx,makeidx,amsthm}
\usepackage{makeidx,graphicx}
\usepackage{amsthm,amsmath,amscd}   
\usepackage{amsfonts,amssymb}       
\usepackage{mathrsfs}
\usepackage{cite}  
\usepackage{geometry} 
\usepackage{mathrsfs,amssymb}
\usepackage{cite}  
\usepackage{indentfirst} 
\usepackage{xcolor}
\usepackage[useregional]{datetime2}

\makeatletter
\renewcommand{\@biblabel}[1]{#1.} 
\makeatother

\newtheorem{theorem}{Theorem}
\newtheorem{lemma}{Lemma}
\newtheorem{prop}{Proposition}
\newtheorem{example}{Example}

\newtheorem{corollary}{Corollary}
\newtheorem{remark}{Remark}

\begin{document}

\begin{center}
	\Large{{\bf Matrix approach to the fractional calculus}}\\
\vskip 0.2 cm
	\large{{\it V. N. Kolokoltsov$^{1,2,3}$}}\\
	{v.n.kolokoltsov@gmail.com}\\

		\large{{\it E. L. Shishkina$^{4,5,6}$}}\\
	{shishkina@amm.vsu.ru}\\
	
	\vskip 0.2 cm
	\small{These authors contributed equally to this work.}
	\vskip 0.2 cm
\end{center}

$^1$Department of Mathematical Statistics, Lomonosov Moscow State University, Leninskie Gory, 1, Moscow, 119991,  Russia\\

$^2$Moscow Center of Fundamental and Applied Mathematics, Lomonosov Moscow State University, Leninskie Gory, 2, Moscow, 119991,  Russia\\

$^3$International Laboratory of Stochastic Analysis and its Applications, National Research University Higher School of Economics, Pokrovsky Bulvar, 11, Moscow, 109028,  Russia\\

$^4$Department of Mathematical and Applied Analysis, Voronezh State University, Universitetskaya pl., 1, Voronezh, 394018,  Russia\\

$^5$Department of Applied Mathematics	and Computer Modeling, Belgorod State National Research University (BelGU), Pobedy St., 85, Belgorod, 308015,  Russia\\

$^6$Institute of Mathematics, Physics and Information Technology,  Kadyrov Chechen State University, A. Sheripova st., 32, Grozny, 364024,  Russia \\

{\bf Abstract.} In this paper, we introduce the new construction of fractional derivatives and integrals with respect to a function, based on a matrix approach. We believe that this is a powerful tool in both analytical and numerical calculations. We begin with the differential operator with respect to a function that generates a semigroup. By discretizing this operator, we obtain a matrix approximation. Importantly, this discretization provides not only an approximating operator but also an approximating semigroup. This point motivates our approach, as we then apply Balakrishnan's representations of fractional powers of operators, which are based on semigroups. Using estimates of the semigroup norm and the norm of the difference between the operator and its matrix approximation, we derive the convergence rate for the approximation of the fractional power of  operators with the fractional power of  correspondings matrix operators. In addition, an explicit formula for calculating an arbitrary power of a two-band matrix is obtained, which is indispensable in the numerical solution of fractional differential and integral equations.




\section{Introduction}\label{sec1}

It is well-established that numerous methods exist for defining fractional derivatives and integrals (see, for example, \cite{SamkKilbMari1993,Kolokoltsov2019,Hilfer,MariShish2024}). The semigroup method offers the most comprehensive, holistic, and universal approach to introducing fractional powers of operators. Fractional powers of the negative infinitesimal generators of bounded strongly continuous semigroups have been explored by Bochner \cite{Bochner}, Phillips \cite{Phillips}, Yosida \cite{YosidaFP}, and Balakrishnan \cite{Balakrishnan1959}.

Let $\{T_t\}_{t\geq 0}$ be a contraction semigroup of type $C_0$ (meaning it is a strongly continuous semigroup) on a Banach space $X$ with infinitesimal generator $(A,D(A))$. 
According to Balakrishnan's approach, see \cite{Balakrishnan1959,Westphal1970,Yosida1980}, the fractional power $(-A)^\alpha$, $0{\,<\,}\alpha{\,<\,}1$ is defined by the formula
\begin{equation}\label{BF00}
	(-A)^\alpha u=\frac{1}{\Gamma(-\alpha)}\int\limits_0^\infty t^{-\alpha-1}(T_t-I)udt, \qquad u\in D(A).
\end{equation}

This method for constructing fractional powers of operators is both highly effective and quite general. In practical applications, we often need a formula for the fractional power of the operator that can be applied in both analytical and numerical calculations. Following this idea,  we present the matrix approach to the fractional calculus.

Consider the operator
\begin{equation}\label{Amin01}
	Af= f_{g}'(x)= \frac{df}{dg} = \lim\limits_{h\to 0}\frac{f(x+h)-f(x)}{g(x+h)-g(x)}=\frac{f'}{g'}.
\end{equation}
Operator \eqref{Amin01} is  a  {\bf derivative of a function $f$ with respect to a function $g$}.  
Such operators were studied in details in \cite{Alfonso}.
 Fractional powers of operators \eqref{Amin01} arise  in different probabilistic problems \cite{Kolokoltsov2019,Orsinger1,Orsinger2}.
 For the full history of fractional calculus with respect to functions see \cite{Arran}.

If $g$ is a strictly increasing function  
having a continuous derivative then one way to define a fractional power of \eqref{Amin01} is (see \cite{SamkKilbMari1993}, p. 326, formula 18.29)
\begin{equation}\label{Eq02}
	(D_{a+,g}^{\alpha}f)(x)=\frac{1}{\Gamma(1-\alpha)}\frac{1}{g'(x)}\frac{d}{dx}\int\limits_a^x \frac{f(t)dg(t)}{(g(x)-g(t))^{\alpha}},\qquad  0<\alpha<1.
\end{equation}
The expression \eqref{Eq02} is called the  {\bf Riemann-Liouville fractional derivatives 
	of function $f$ by a function $g$ of the order $\alpha$} ($0<\alpha<1$).

In this article our main idea is to use powers of matrices in order to obtain
fractional powers    of \eqref{Amin01}.
It is worth noting that the idea of using matrices in differential and integral calculus was  proposed in a simple specific case by Wiliam Swartz in \cite{WiliamSwartz}. Swartz suggested using matrix inversion to integrate certain functions, such as $e^{ax}\sin(bx)$, $e^{ax}\cos(bx)$, $x^ne^x$ . Specifically, for the integral $\int x^n e^xdx$, he constructed the vector of derivatives of $x^ne^x$ for $n=0,1,2,...,n$  in the form $(e^x,e^x+xe^x,...,nx^{n-1}e^x+x^ne^x)$, created corresponding matrix, inversed it and obtained integrals.

In the pioneering work \cite{PodlubnyMatr01} of I. Podlubny the matrix form representation of discrete analogues of various forms of fractional differentiation and fractional integration was suggested. Elements of matrices in \cite{PodlubnyMatr01} were constructed by using the backward fractional difference approximation for the $\alpha$-th derivative.  Matrix approach to discrete fractional calculus developed in \cite{Podlubny,PodlubnyMatr01,PodlubnyMatr02,PodlubnyMatr03} was efficiently applied to numerical solution of fractional differential and integral equations.

In our article, we consider the derivative of one function with respect to another and first construct the fractional power of the corresponding matrix. We then obtain an approximation of the fractional derivative of order $\alpha$.
	The approximation is achieved by replacing the original operator with a matrix operator and analyzing the associated semigroups to derive optimal convergence estimates for the fractional powers of the operator.
	The main result consists of convergence rate estimates for this approximation method.
	It is presented in the fifth section.
	Applications of the main result to   concrete operators are given in  the seventh section.
	 Our approach is based on the general theory of operators and semigroups as presented in \cite{KolokoltsovBook}.

In the second section, we provide an illustrative example explaining the relationship between fractional powers of a matrix and fractional powers of the differentiation operator. The third section derives an explicit formula for computing any arbitrary natural power of a double band matrix. In the fourth section, Balakrishnan's approach to constructing fractional powers of operators is examined in detail. The fifth section presents an estimate of the convergence rate of the approximate operator's power to that of the original operator. The sixth section explores the derivative of one function with respect to another. The seventh section establishes an estimate of the convergence rate  of the matrix operator for such derivatives in the case of strictly increasing measures. Finally, in the eighth section we present an arbitrary power of two-band matrices which is enabled the direct construction of fractional powers of operators, facilitating their use in computational applications.
 This article presents an alternative matrix approach to construction of the fractional power of the derivative of one function with respect to another compared to the method implemented in \cite{KolokoltsovShishkina}.  

\section{Illustrative example: matrix fractional derivative}\label{Sec02}

To illustrate the concept of using matrices for introducing new forms of fractional integro-differentiation, we begin with the simplest case and consider a derivative.
$(\mathcal{D}f)(x){\,=\,}\frac{df}{dx}$, where  $x{\,\in\,}[a,b]$, $f{\,\in\,}C([a,b])$, $a,b{\,\in\,}{\mathbb{R}}$,  $f(a){\,=\,}0$.

Let $P'{\,=\,}(x_{0},x_{1},\ldots ,x_{n})$ be a partition of $[a,x]$, $a<x\leq b$ that is
\begin{align*}
	a=x_{0}<x_{1}<x_{2}<\dots <x_{n}&=x,\\
	x_k-x_{k-1}=h=\frac{x-a}{n}&>0,
\end{align*}  
so 
$x_0{\,=\,}a$, $x_1{\,=\,}x_0{\,+\,}h$, $x_2{\,=\,}x_0{\,+\,}2h$,  $...$, $x_{n-1}{\,=\,}x_0{\,+\,}(n-1)h$, $x_{n}{\,=\,}x_0{\,+\,}nh{\,=\,}x$.

The partition $P'$ corresponds to the left-hand derivative $\mathcal{D}_{\ell}$ at the point $x{\,=\,}x_n$, $n{\,=\,}1,...,N$ is defined as  the limit of the  backward finite difference
\begin{equation}\label{Derl}
	(\mathcal{D}_{\ell}f)(x)=\lim\limits_{h\rightarrow
		0} \frac{f(x)-f(x-h)}{h}=f'_{\ell}(x)
\end{equation}
It we consider $[x,b]$, $a\leq x<b$, $h>0$ we take a partition $P''=(x_{0},x_{1},\ldots ,x_{n})$ of $[x,b]$, $a\leq x<b$
that is
\begin{align*}
	x=x_{0}<x_{1}<x_{2}<\dots <x_{n}&=b,\\
	x_k-x_{k-1}=h=\frac{b-x}{n}&>0,
\end{align*}  
so 
$x_0{\,=\,}x$, $x_1{\,=\,}x_0{\,+\,}h$, $x_2{\,=\,}x_0{\,+\,}2h$, $...$, $x_{n-1}{\,=\,}x_0{\,+\,}(n-1)h$, $x_{n}{\,=\,}x_0{\,+\,}nh{\,=\,}b$.

The partition $P''$ corresponds to
the right-hand derivative $\mathcal{D}_{r}$ at the point $x$ is
\begin{equation}\label{Derr}
	(\mathcal{D}_{r}f)(x)=\lim\limits_{h\rightarrow
		0} \frac{f(x+h)-f(x)}{h}=f'_{r}(x).
\end{equation}

Let $I_{n}=\underbrace{(1,0,...,0)}_{n}$.
In order to define the left-hand derivative in matrix form, we  set  
\begin{align*}
	\,_{\ell}f_0&=f(x_n)=f(x),\\
	\,_{\ell}f_1&=f(x_{n-1})=f(x_n-h),\\
	&...,\\
	\,_{\ell}f_{n-1}&=f(x_1)=f(x_2-h),\\
	\,_{\ell}f_{n}&=f(x_0)=f(x_1-h)=f(a)=0,\\
	\,_{\ell}\mathbf{f}_n'&=(\,_{\ell}f_0,\,_{\ell}f_{1},\,_{\ell}f_{2},...,\,_{\ell}f_{n-1})\quad{\text{is a vector of length}}\quad n.
\end{align*}
Then we associate the operator $\mathcal{D}_{\ell}$  with the $(n{\times}n)$ matrix
\begin{equation}\label{ed}
	\mathcal{A}_n=\frac{1}{h}\begin{pmatrix}
		1    & -1        & 0       & \cdots  & 0 & 0\\
		0    & 1         & -1      & \cdots & 0 & 0 \\
		\vdots   &	\vdots   &  \vdots & \cdots & \vdots & \vdots \\
		0    & 0         & 0       & \cdots & 1 & -1 \\
		0    & 0         & 0       & \cdots & 0  & 1
	\end{pmatrix},
\end{equation} 
in the sense
$$
\lim\limits_{h\rightarrow 0} \mathcal{A}_n\cdot \,_{\ell}\mathbf{f}_n=
$$
$$
=\lim\limits_{h\rightarrow
	0}\frac{1}{h}\begin{pmatrix}
	1    & -1        & 0       & \cdots  & 0 & 0\\
	0    & 1         & -1      & \cdots & 0 & 0 \\
	\vdots   &	\vdots   &  \vdots & \cdots & \vdots & \vdots \\
	0    & 0         & 0       & \cdots & 1 & -1 \\
	0    & 0         & 0       & \cdots & 0  & 1
\end{pmatrix} \cdot \begin{pmatrix}
	\,_{\ell}f_{0} \\
	\,_{\ell}f_{1} \\
	\vdots    \\
	\,_{\ell}f_{n-2} \\
	\,_{\ell}f_{n-1}
\end{pmatrix} =\lim\limits_{h\rightarrow
	0}\frac{1}{h} \begin{pmatrix}
	\,_{\ell}f_{0}-\,_{\ell}f_{1}  \\
	\,_{\ell}f_{1}-\,_{\ell}f_{2}  \\
	\vdots    \\
	\,_{\ell}f_{n-2}-\,_{\ell}f_{n-1} \\
	\,_{\ell}f_{n-1}-\,_{\ell}f_{n}
\end{pmatrix}=
$$
$$
=\begin{pmatrix}
	\lim\limits_{h\rightarrow
		0}\frac{f(x_{n})-f(x_{n-1})}{h}  \\
	\lim\limits_{h\rightarrow
		0}\frac{f(x_{n-1})-f(x_{n-2}) }{h} \\
	\vdots    \\
	\lim\limits_{h\rightarrow
		0}\frac{f(x_{2})-f(x_{1})}{h}          \\
	\lim\limits_{h\rightarrow
		0}\frac{f(x_1)-f(x_{0})}{h}  
\end{pmatrix}= \begin{pmatrix}
	f'_{\ell}(x_{n})  \\
	f'_{\ell}(x_{n-1})\\
	\vdots    \\
	f'_{\ell}(x_{2})          \\
	f'_{\ell}(x_{1})   
\end{pmatrix}
$$
and
$$
	(\mathcal{D}_{\ell}f)(x)=\lim\limits_{h\rightarrow 0} (\mathcal{A}_n  \cdot \,_{\ell}\mathbf{f}_n)\cdot I_{1} =
	\lim\limits_{h\rightarrow
		0} \frac{f(x_n)-f(x_{n-1})}{h}=f'_{\ell}(x)
$$
is a  derivative of $f$ at $x=x_n$.

In order to define the right-hand derivative in matrix form, we  set  
\begin{align*}
	\,_{r}f_0&=f(x_0)=f(x),\\
	\,_{r}f_1&=f(x_{1})=f(x_0+h),\\
	&...,\\
	\,_{r}f_{n-1}&=f(x_{n-1})=f(x_{n-2}+h),\\
	\,_{r}f_{n}&=f(x_n)=f(x_{n-1}+h)=f(b)=0,\\
	\,_{r}\mathbf{f}_n'&=(\,_{r}f_0,\,_{r}f_{1},\,_{r}f_{2},...,\,_{r}f_{n-1})\quad{\text{is a vector of length}}\quad n.
\end{align*}
We associate the operator $\mathcal{D}_{r}$  with the $(n{\times}n)$ matrix $\mathcal{A}_n$ in the sense
$$
\lim\limits_{h\rightarrow 0} \mathcal{A}_n\cdot (-\,_{r}\mathbf{f}_n)=
$$
$$
=\lim\limits_{h\rightarrow
	0}\frac{1}{h}\begin{pmatrix}
	1    & -1        & 0       & \cdots  & 0 & 0\\
	0    & 1         & -1      & \cdots & 0 & 0 \\
	\vdots   &	\vdots   &  \vdots & \cdots & \vdots & \vdots \\
	0    & 0         & 0       & \cdots & 1 & -1 \\
	0    & 0         & 0       & \cdots & 0  & 1
\end{pmatrix} \cdot \begin{pmatrix}
	-\,_{r}f_{0} \\
	-\,_{r}f_{1} \\
	\vdots    \\
	-\,_{r}f_{n-2} \\
	-\,_{r}f_{n-1}
\end{pmatrix} =\lim\limits_{h\rightarrow
	0}\frac{1}{h} \begin{pmatrix}
	\,_{r}f_{1}-\,_{r}f_{0}  \\
	\,_{r}f_{2}-\,_{r}f_{1}  \\
	\vdots    \\
	\,_{r}f_{n-1}-\,_{r}f_{n-2} \\
	\,_{r}f_{n}-\,_{r}f_{n-1}
\end{pmatrix}=
$$
$$
=\begin{pmatrix}
	\lim\limits_{h\rightarrow
		0}\frac{f(x_{1})-f(x_{0})}{h}  \\
	\lim\limits_{h\rightarrow
		0}\frac{f(x_{2})-f(x_{1}) }{h} \\
	\vdots    \\
	\lim\limits_{h\rightarrow
		0}\frac{f(x_{n-1})-f(x_{n-2})}{h}          \\
	\lim\limits_{h\rightarrow
		0}\frac{f(x_n)-f(x_{n-1})}{h}  
\end{pmatrix}= \begin{pmatrix}
	f'_{r}(x_{0})  \\
	f'_{r}(x_{1})\\
	\vdots    \\
	f'_{r}(x_{n-2})          \\
	f'_{r}(x_{n-1})   
\end{pmatrix}
$$
and
$$
	(\mathcal{D}_{r}f)(x)=\lim\limits_{h\rightarrow 0} (\mathcal{A}_n  \cdot (-\,_{r}\mathbf{f}_n))\cdot I_{n} =
	\lim\limits_{h\rightarrow
		0} \frac{f(x_1)-f(x_{0})}{h}=f'_{r}(x)
$$
is a  derivative of $f$ at $x=x_0$.

We  explore the concept of associating   differential operators with its corresponding matrix. Our goal is to construct the fractional power  $\mathcal{A}^\alpha_n$, $\alpha{\in}\mathbb{R}$, of the matrix operator $\mathcal{A}_n$. Subsequently, we will derive the fractional powers of   operators $\mathcal{D}_{\ell}$ and $\mathcal{D}_{r}$ as a limit by the formulas, respectively 
\begin{equation}\label{FracPow}
	(\mathcal{D}^\alpha_{\ell} f)(x)=
	\lim\limits_{h\to 0} (\mathcal{A}^{\alpha}_n \cdot \,_{\ell}\mathbf{f}_n)\cdot I_{n},\qquad 	(\mathcal{D}^\alpha_{r} f)(x)=
	\lim\limits_{h\to 0} (\mathcal{A}^{\alpha}_n \cdot (-\,_{r}\mathbf{f}_n))\cdot I_{n}
\end{equation}
for all $\alpha\in \mathbb{R}$. 

In order to construct the power $\mathcal{A}^\alpha_n$, $\alpha{\in}\mathbb{R}$ we will use the formula $\mathcal{A}^\alpha_n{=}e^{\alpha \ln A_n}$. Choosing for uniqueness  the main branch of the logarithm $\ln \mathcal{A}_n$, we obtain
$$
\ln \mathcal{A}_n=-\ln h+\begin{pmatrix}
	0               & -1       & -\frac{1}{2}   &\cdots   & -\frac{1}{n-2}    & -\frac{1}{n-1}\\
	0               & 0       &        -1     & \cdots   & -\frac{1}{n-3}     & -\frac{1}{n-2} \\
	\vdots          &  \vdots & \vdots       & \cdots   & \vdots & \vdots \\
	0           & 0       &        0     & \cdots   & 0              & -1 \\
	0               & 0       &        0     & \cdots   & 0               & 0 \\
\end{pmatrix},
$$
$$
\alpha\ln \mathcal{A}_n=-\alpha\ln h+\begin{pmatrix}
	0               & -\alpha       & -\frac{\alpha}{2}   &\cdots   & -\frac{\alpha}{n-2}    & -\frac{\alpha}{n-1}\\
	0               & 0       &       -\alpha     & \cdots   & -\frac{\alpha}{n-3}     & -\frac{\alpha}{n-2} \\
	\vdots          &  \vdots & \vdots       & \cdots   & \vdots & \vdots \\
	0           & 0       &        0     & \cdots   & 0              & -\alpha \\
	0               & 0       &        0     & \cdots   & 0               & 0 \\
\end{pmatrix},
$$
\begin{equation}\label{FracPow0}
	\mathcal{A}^\alpha_n=\frac{1}{h^\alpha}\begin{pmatrix}
		1         & -\alpha &     	\frac{\alpha(\alpha-1)}{2}   & \cdots & (-1)^{n-2}\frac{\alpha(\alpha-1)...(\alpha-n+3)}{(n-2)!} & (-1)^{n-1}\frac{\alpha(\alpha-1)...(\alpha-n+2)}{(n-1)!}\\
		0 &      1         & -\alpha &     	 \cdots & (-1)^{n-3}\frac{\alpha(\alpha-1)...(\alpha-n+4)}{(n-3)!} & (-1)^{n-2}\frac{\alpha(\alpha-1)...(\alpha-n+3)}{(n-2)!}\\
		\vdots   &  \vdots & \vdots & \cdots & \vdots & \vdots\\
		0   & 0   & 0&\cdots & 1 & -\alpha\\
		0   & 0   & 0 &  \cdots & 0 & 1
	\end{pmatrix}.
\end{equation}
We get general power \eqref{FracPow0} of  matrix operator $\mathcal{A}_n$ for all $\alpha\in \mathbb{R}$ which automatically has semigroup property.
We will call $\mathcal{A}^\alpha_n$ by {\bf matrix fractional derivative}.

Next we notice that when $\alpha=m$, $m\in\mathbb{N}$, $m\leq n$ we get that  in the first row all elements with a column number greater or equal than $(m{+}2)$ are equal to zero
\begin{equation}\label{FracPow00}
	\mathcal{A}^{m}_n=\frac{1}{h^m}\begin{pmatrix}
		1        &-m         &   \frac{m(m-1)}{2}   & \cdots & (-1)^m\frac{m(m-1)...1}{m!}     & 0         & \cdots &   0 \\
		0        &  1        & -m                     & \cdots & (-1)^{m-1}\frac{m(m-1)...2}{(m+1)!} & 0         & \cdots & 0 \\
		\vdots   &  \vdots   & \vdots                 & \cdots & \vdots                             &  \vdots   &\vdots  &\vdots\\
		0        & 0         & 0                      &\cdots  & 0                                  & 1         & -m     &   \frac{m(m-1)}{2}\\
		0        & 0         & 0                      &\cdots  & 0                                  & 0         &1       & -m\\
		0        & 0         & 0                      & \cdots & 0                                  & 0         &0       & 1
	\end{pmatrix},
\end{equation}
so for $\alpha=m$, $m\in\mathbb{N}$, $m\leq n$ 
we get
$$
\mathcal{A}^{m}_n\cdot \,_{\ell}\mathbf{f}_n=\frac{1}{h^m}\begin{pmatrix}
	1        &-m         &   \frac{m(m-1)}{2}   & \cdots & (-1)^m\frac{m(m-1)...1}{m!}     & 0         & \cdots &   0 \\
	0        &  1        & -m                     & \cdots & (-1)^{m-1}\frac{m(m-1)...2}{(m+1)!} & 0         & \cdots & 0 \\
	\vdots   &  \vdots   & \vdots                 & \cdots & \vdots                             &  \vdots   &\vdots  &\vdots\\
	0        & 0         & 0                      &\cdots  & 0                                  & 1         & -m     &   \frac{m(m-1)}{2}\\
	0        & 0         & 0                      &\cdots  & 0                                  & 0         &1       & -m\\
	0        & 0         & 0                      & \cdots & 0                                  & 0         &0       & 1
\end{pmatrix}\cdot \begin{pmatrix}
	\,_{\ell}f_{0} \\
	\,_{\ell}f_{1} \\
	\vdots    \\
	\,_{\ell}f_{n-3}\\
	\,_{\ell}f_{n-2} \\
	\,_{\ell}f_{n-1}
\end{pmatrix} =
$$
$$
=\frac{1}{h^m}\begin{pmatrix}
	\,_{\ell}f_{0}-m\cdot\,_{\ell}f_{1}+\frac{m(m-1)}{2}\cdot\,_{\ell}f_{2}+\cdots+ (-1)^m\frac{m(m-1)...1}{m!}\cdot\,_{\ell}f_{m} \\
	\,_{\ell}f_{1}-m\cdot\,_{\ell}f_{2}+  \cdots + (-1)^{m-1}\frac{m(m-1)...2}{(m+1)!}\cdot\,_{\ell}f_{m-2} \\
	\vdots    \\
	\,_{\ell}f_{n-3} -m\cdot\,_{\ell}f_{n-2}+\frac{m(m-1)}{2}\cdot \,_{\ell}f_{n-1}\\
	\,_{\ell}f_{n-2}-m\cdot\,_{\ell}f_{n-1} \\
	\,_{\ell}f_{n-1}
\end{pmatrix}
$$
and 
\begin{multline*}
(\mathcal{D}^{m}_{\ell}f)(x)=\lim\limits_{h\rightarrow
	0} (\mathcal{A}_n^{m} \cdot \,_{\ell}\mathbf{f}_n)\cdot I_{n} 
=\lim\limits_{h\rightarrow
	0}\sum\limits_{k=0}^{m}(-1)^{k}\binom{m}{k}f(x_n-kh)=\\
	=\lim\limits_{h\rightarrow
	0}\frac{\bigtriangledown_{h}^{m}[f](x)}{h^m}= f^{(m)}_{\ell}(x), 
\end{multline*}
where $\binom{m}{k}{=}\frac{m!}{k!(m-k)!}$ are binomial coefficients,
$ \nabla_{h}^{m}[f](x){=}\sum\limits_{k=0}^{m}(-1)^{k}\binom{m}{k}f(x-kh)$ is a  $m$-th order backward finite difference,
$f^{(m)}_{\ell}(x)$ is a left-handed derivative of $f$ at $x{\,=\,}x_n$ of order $m\in\mathbb{N}$.

Analogically,
\begin{multline*}
(\mathcal{D}^{m}_{r}f)(x)=\lim\limits_{h\rightarrow
	0} (\mathcal{A}_n^{m} \cdot (-1)^m\,_{r}\mathbf{f}_n)\cdot I_{n}=\\ 
=\lim\limits_{h\rightarrow
	0}\sum\limits_{k=0}^{m}(-1)^{m+k}\binom{m}{k}f(x_0+kh)=\lim\limits_{h\rightarrow
	0}\frac{\bigtriangleup_{h}^{m}[f](x)}{h^m}= f^{(m)}_{\ell}(x), 
\end{multline*}
where $\bigtriangleup_{h}^{m}[f](x){=}\sum\limits_{k=0}^{m}(-1)^{m+k}\binom{m}{k}f(x_0+kh)$
is a  $m$-th order forward finite difference,
$f^{(m)}_{r}(x)$ is a right-handed derivative of $f$ at $x{\,=\,}x_0$ of order $m\in\mathbb{N}$.

Now we take $\alpha\in\mathbb{R}$, $\alpha>0$, $\alpha\neq m$, $m\in\mathbb{N}$.
Taking into account that 
$$
\alpha(\alpha-1)...(\alpha-k+1)=(\alpha)_{k}
$$ 
is	the falling factorial which can be written  using the Gamma function by the formula
$$ 
(\alpha)_{k}=\frac {\Gamma (\alpha+1)}{\Gamma (\alpha-k+1)},
$$
and noticing that
the falling   factorial is connected with   a binomial coefficient by the formula: 
$$
\frac {(\alpha)_{k}}{k!}={\binom{\alpha}{k}}
$$
we can rewrite \eqref{FracPow0} in the form
\begin{equation}\label{MGLI}
	\mathcal{A}^\alpha_n=\frac{1}{h^\alpha}\begin{pmatrix}
		{\binom{\alpha}{0}}         & -{\binom{\alpha}{1}}  & {\binom{\alpha}{2}} &     \cdots & (-1)^{n-1}{\binom{\alpha}{n-1}}\\
		0 &     {\binom{\alpha}{0}}         & -{\binom{\alpha}{1}} &     	 \cdots  &(-1)^{n-2}{\binom{\alpha}{n-2}}\\
		\vdots   &  \vdots & \vdots & \cdots &   \vdots\\
		0   & 0   & 0 &  \cdots &   {\binom{\alpha}{0}}  
	\end{pmatrix}.
\end{equation}
Therefore, by \eqref{FracPow}, taking into account that $n\to\infty$ when $h\to 0$, we can write positive fractional power $\alpha$ of \eqref{Derl}:
\begin{multline}\label{GLD}
	(\mathcal{D}^\alpha_{\ell} f)(x)=\lim\limits_{h\rightarrow
		0} (\mathcal{A}_n^{\alpha} \cdot \mathbf{f}_n)\cdot I_{n}=\lim\limits_{h\rightarrow
		0}
	\frac{1}{h^\alpha}\sum\limits_{k=0}^{\infty}(-1)^{k}\frac{\Gamma(\alpha+1)}{k!\Gamma(\alpha-k+1)}f(x-kh)=\\
	=\lim\limits_{h\rightarrow
		0}
	\frac{1}{h^\alpha}\sum\limits_{k=0}^{\infty}(-1)^{k}\frac{(\alpha)_{k}}{k!}f(x-kh)=\lim\limits_{h\rightarrow
		0}
	\frac{1}{h^\alpha}\sum\limits_{k=0}^{\infty}(-1)^{k}\binom{\alpha}{k}f(x-kh).
\end{multline}	
Formula \eqref{GLD} gives the  left-hand Gr\"{u}nwald-Letnikov fractional derivative of order  $\alpha{\,>\,}0$ (see \cite{SamkKilbMari1993}, p. 373, formula (20.7)).

Taking $(-\alpha)$ instead of $\alpha$ and noticing that 
$$
(-1)^{k}\binom{-\alpha}{k}=(-1)^k\frac{(-\alpha)_k}{k!}=\frac{(\alpha+k-1)_k}{k!}=\frac{\Gamma(\alpha+k)}{k!\Gamma(\alpha)}=\frac{\alpha^{(k)}}{k!},
$$
where $\alpha^{(k)}=\frac{\Gamma(\alpha+k)}{\Gamma(\alpha)}$ is a rising factorial,
we obtain from \eqref{GLD}  the left-sided Gr\"{u}nwald-Letnikov integral of order $\alpha>0$
\begin{multline}\label{GLI}
	(\mathcal{I}^\alpha_{\ell}f)(x)=	(\mathcal{D}^{-\alpha}_{\ell}f)(x)=\lim\limits_{h\rightarrow
		0}
	h^\alpha\sum\limits_{k=0}^{\infty}(-1)^{k}\binom{-\alpha}{k}f\left(x-kh\right)=\\
	=\lim\limits_{h\rightarrow
		0}
	h^\alpha\sum\limits_{k=0}^{\infty}\frac{\Gamma(\alpha+k)}{k!\Gamma(\alpha)}f\left(x-kh\right)=\\
	=\lim\limits_{h\rightarrow
		0}
	h^\alpha\sum\limits_{k=0}^{\infty}\frac{\alpha^{(k)}}{k!}f\left(x-kh\right)=\lim\limits_{h\rightarrow
		0}
	h^\alpha\sum\limits_{k=0}^{\infty}\binom{\alpha+k-1}{k}f\left(x-kh\right).
\end{multline}
Combining formula \eqref{GLD} and \eqref{GLI} we get 	 the left-sided Gr\"{u}nwald-Letnikov fractional integro-differentiation of order  $\alpha\in \mathbb{R}$,
$\alpha\neq-m$, $m\in\mathbb{N}$ (see  \cite{SamkKilbMari1993}).

Considering $[x,b]$ by the same way we can  obtain the  right-sided Gr\"{u}nwald-Letnikov fractional derivative and integral of \eqref{Derr} of order  $\alpha>0$
\begin{equation}\label{GLD01}
	(\mathcal{D}^\alpha_{r} f)(x){\,=\,}\lim\limits_{h\rightarrow
		0}
	\frac{e^{i\alpha\pi}}{h^\alpha}\sum\limits_{k=0}^{\infty}(-1)^{k}\frac {(\alpha)_{k}}{k!}f(x+kh)=\lim\limits_{h\rightarrow0}  
	\frac{e^{i\alpha\pi}}{h^\alpha}\sum\limits_{k=0}^{\infty}(-1)^{k}\binom{\alpha}{j}f(x+kh),
\end{equation}
\begin{equation}\label{GLI01}
	(\mathcal{I}^\alpha_{r}f)(x){\,=\,}\lim\limits_{h\rightarrow
		0}
	e^{-i\alpha\pi}h^\alpha\sum\limits_{k=0}^{\infty}\frac{\alpha^{(k)}}{k!}f\left(x+kh\right)=\lim\limits_{h\rightarrow
		0}
	e^{-i\alpha\pi}h^\alpha\sum\limits_{k=0}^{\infty}\binom{\alpha{+}k{-}1}{k}f\left(x+kh\right),
\end{equation}
where we take $(-1)^\alpha=e^{i\alpha\pi}$.

In \cite{SamkKilbMari1993} fractional order differences were considered. We will use notations
$$
\nabla_{h}^{\alpha}[f](x)=(I-\,_{\ell}T_h)^\alpha f=\sum\limits_{k=0}^{\infty}(-1)^{k}\binom{\alpha}{k}f(x-kh),\qquad \,_{\ell}T_hf(x)=f(x-h),
$$ 
which is left-sided finite difference when $h>0$ and 
$$
\bigtriangleup_{h}^{\alpha}[f](x)=(\,_{r}T_h-I)^\alpha f=e^{i\alpha\pi}\sum\limits_{k=0}^{\infty}(-1)^{k}\binom{\alpha}{k}f(x+kh),\qquad \,_{r}T_hf(x)=f(x+h),
$$ 
which is right-sided finite difference when $h>0$.

Thus, we can write  general forms of the finite-difference operator of fractional integro-differentiation for $h>0$:
$$
f^{(\alpha)}_{\ell}(x)=\lim\limits_{h\to 0}\frac{\nabla_{h}^{\alpha}[f](x)}{h^\alpha},\qquad f^{(\alpha)}_{r}(x)=\lim\limits_{h\to 0}\frac{ \bigtriangleup_{h}^{\alpha}[f](x)}{h^\alpha},\qquad \alpha\in\mathbb{R}.
$$
When $\alpha={\,m\,}\in\mathbb{N}\cup\{0\}$ we should use 
$$
\nabla_{h}^{m}[f](x){\,=\,}\sum\limits_{k=0}^{m}(-1)^{k}\binom{m}{k}f(x- kh)\quad \text{and}\quad \bigtriangleup_{h}^{m}[f](x){\,=\,}\sum\limits_{k=0}^{m}(-1)^{k+m}\binom{m}{k}f(x+kh).
$$

\begin{remark} We obtained the Gr\"{u}nwald-Letnikov derivative using  the formula for a degree of the matrix. Next, we generalize this approach in order to obtain the fractional power of the operator $\frac{d}{dg}$, $g{\,=\,}g(x)$ in the finite difference  form.
\end{remark}

\begin{remark} It is known that (see \cite{Podlubny})
	for  $f\in L_1(a,b)$ limit \eqref{GLI} exists 
	for almost everyone $x$ and
	$(\mathcal{I}^{\alpha}_{\ell}f)(x)=({I}^\alpha_{a+}f)(x)$,  $(\mathcal{D}^{\alpha}_{\ell}f)(x)=({D}^\alpha_{a+}f)(x),$
	where $I^\alpha_{a+}$ is a left-sided Riemann-Liouville integral, and $D_{a+}^{\alpha}$
	is a left-sided Riemann-Liouville derivative \cite{SamkKilbMari1993}.
\end{remark}

So we constructed a fractional powers of $\mathcal{D}_{\ell}$ and $\mathcal{D}_{r}$ using a real power $\alpha$ of matrix $\mathcal{A}_n$ and obtained the   derivative of order $m$  when $\alpha{\,=\,}-m$, $m\in\mathbb{N}$, $m{\,\leq\,}n$ and  the   Gr\"{u}nwald-Letnikov   derivative, when $\alpha\neq-m$, $m\in\mathbb{N}$. Our goal now is to generalise this approach and construct a fractional power of $\frac{df}{dg}$ in the finite difference form, where $g$ is an appropriate function. 

\section{Natural power of  two-band matrices}\label{NP}

We want to combine Balakrishnan's approach to the power of an operator \eqref{BF00} and matrix approach. As we saw in formula \eqref{BF00}  in order to construct the fractional power of the operator $(-A)$ we need to know its semigroup 
$$T_t=e^{-tA}=\sum\limits_{k=0}^\infty \frac{(-t)^k}{k!}A^k.$$ 
So we start from  obtaining the power  $A^k$ of two-band matrix, generalizing
\eqref{ed}.  Natural powers of some special upper triangular matrices were constructed in \cite{Khetchatturat}. 

The complete homogeneous symmetric polynomial of degree $q$ in $m$ variables $a_1,...,a_m$ 
written $\mathscr{H}_q(a_1,...,a_m)$ for $q=0,1,2,3,...$  is defined by \cite{Khetchatturat}
$$
\mathscr{H}_q(a_1,a_2,...,a_m)=\sum\limits_{i_1+i_2+...+i_m=q}a_1^{i_1}\cdot a_2^{i_2}\cdot...\cdot a_m^{i_m},
$$
where $i_p\geq 0$, $p=1,...,m$. When an exponent is zero, the corresponding power variable is taken to be $1$.

Another form of the complete homogeneous symmetric polynomial of degree $k$ in $m$ variables is the sum of all monomials of total
degree $k$ in the variables
$$
\mathscr{H}_q(a_1,a_2,...,a_m)=\sum\limits_{1\leq i_1\leq i_2\leq...\leq i_q\leq m}a_{i_1}\cdot a_{i_2}\cdot...\cdot a_{i_q}.
$$

Using results from \cite{Khetchatturat,Cornelius,Bhatnagar} we can write  the complete homogeneous symmetric polynomial 
$\mathscr{H}_{k-m+s}(a_s,a_{s+1}...,a_m)$, $s<m$
as the sum of rational functions:
\begin{equation}\label{tai0}
	\mathscr{H}_{k-m+s}(a_s,a_{s+1}...,a_m)=\sum\limits_{j=s}^m\frac{a_j^k}{\prod\limits_{i=s,j\neq i}^{m}(a_j-a_i)}.
\end{equation}
The identity \eqref{tai0} for $s=1$ is known as Sylvester's identity.

It is known that 
$$
\mathscr{H}_q(1,v,v^2,...,v^{m-1})={m+q-1 \choose q}_{v},
$$ 
where ${p \choose q}_{v}$ is the Gaussian binomial coefficient defined by 
$$
{p \choose r}_{v}={\frac {(1-v^{p})(1-v^{p-1})\cdots (1-v^{p-r+1})}{(1-v)(1-v^{2})\cdots (1-v^{r})}},
$$
where $p$ and $r$ are non-negative integers (see \cite{Exton}).

The evaluation of a Gaussian binomial coefficient at $q=1$ is  
\begin{equation}\label{Binom}
	\lim\limits_{q\to 1}{\binom {p}{r}}_{q}={\binom {p}{r}}=\frac{p(p-1)...(p-r+1)}{r!}. 
\end{equation}

Stirling numbers of the second kind are given by formula
$$
\left\{{n \atop q}\right\}={\frac {1}{q!}}\sum _{i=1}^{q}(-1)^{q-i}{\binom{q}{i}}i^{n}=\sum _{i=1}^{q}{\frac {(-1)^{q-i}}{(q-i)!i!}}i^{n}\qquad 1\leq q\leq n,\qquad n\geq 1.
$$
The evaluation at integers of complete symmetric homogeneous polynomials is related to Stirling numbers of the second kind:
$$
\mathscr{H}_{q}(1,2,\ldots ,m)=\left\{{\begin{matrix}q+m\\m\end{matrix}}\right\}.
$$

\begin{lemma}
	The next formula is valid:
	\begin{equation}\label{tai01}
		\mathscr{H}_{q}(a_1,a_2,...,a_m)=\mathscr{H}_{q}(a_1,a_2,...,a_{m-1})+a_{m}\mathscr{H}_{q-1}(a_1,a_2,...,a_{m}).
	\end{equation}
\end{lemma}	
\begin{proof}
	Using formula \eqref{tai0} we obtain
	\begin{multline*}
		\mathscr{H}_{q}(a_1,a_2,...,a_{m-1})+a_{m}\mathscr{H}_{q-1}(a_1,a_2,...,a_{m})=\\
		=\sum\limits_{\ell=1}^{m-1}\frac{a_\ell^{m+q-2}}{\prod\limits_{s=1,s\neq \ell}^{m-1}(a_\ell-a_s)}+\sum\limits_{\ell=1}^m\frac{a_ma_\ell^{m+q-2}}{\prod\limits_{s=1,s\neq \ell}^{m}(a_\ell-a_s)}=\\
		=\sum\limits_{\ell=1}^{m-1}\frac{a_\ell^{m+q-2}}{\prod\limits_{s=1,s\neq \ell}^{m-1}(a_\ell-a_s)}\left(1+\frac{a_m}{a_\ell-a_m} \right) +\frac{a_m^{m+q-1}}{\prod\limits_{s=1,s\neq m}^{m}(a_m-a_s)}=\\
		=\sum\limits_{\ell=1}^{m-1}\frac{a_\ell^{m+q-1}}{\prod\limits_{s=1,s\neq \ell}^{m}(a_\ell-a_s)} +\frac{a_m^{m+q-1}}{\prod\limits_{s=1,s\neq m}^{m}(a_m-a_s)}=\sum\limits_{\ell=1}^{m}\frac{a_\ell^{m+q-1}}{\prod\limits_{s=1,s\neq \ell}^{m}(a_\ell-a_s)}=\mathscr{H}_{q}(a_1,a_2,...,a_m).
	\end{multline*}
\end{proof}

\begin{theorem}\label{teopowder} 
	Let
	\begin{equation}\label{B}
		{\mathcal A}_n=\begin{pmatrix}
			a_1 & -a_1 & 0          & \cdots & 0 \\
			0 & a_2 & -a_2    & \cdots & 0 \\
			0 & 0 & a_3     & \cdots & 0 \\
			\vdots & \vdots & \vdots    & \ddots & \vdots  \\
			0 & 0      & 0&  \cdots & a_n 
		\end{pmatrix},
	\end{equation}
	then the power $k{\,\in\,}\mathbb{N}\cup\{0\}$ of matrix ${\mathcal A}_n$ is
	\begin{equation}\label{PowMat}
		{\mathcal A}_n^k=(p_{sm})_{m=1,s=1}^{n,n},
	\end{equation}
	where
	\begin{equation}\label{PowMatEl}
		p_{sm}=\left\{ \begin{array}{ll}
			0 & \mbox{if\, $m<s$\, and\, $m>k+s$};\\
			a_s^k & \mbox{if\, $m=s$};\\
			(-1)^{m-s}\cdot \mathscr{H}_{k-m+s}(a_s,a_{s+1},...,a_m)\cdot\prod\limits_{i=s}^{m-1}a_i& \mbox{if\,  $s< m\leq k+s$}.	
		\end{array} \right.
	\end{equation}
\end{theorem}
\begin{proof}
	We will use mathematic induction. For the base step, $k=1$
	$$
	p_{sm}=\left\{ \begin{array}{ll}
		0 & \mbox{if\, $m<s$\, and\, $m>1+s$};\\
		a_s & \mbox{if\, $m=s$};\\
		-a_{s}& \mbox{if\,  $m=1+s$}	
	\end{array} \right.
	$$
	that gives   ${\mathcal A}_n{\,=\,}{\mathcal A}_n^1$.
	Suppose that for $k{\,<\,}n$ formula \eqref{PowMat} is valid. Let us consider $(k+1){\,<\,}n$. Applying formula \eqref{tai01} we obtain
	\begin{multline*}
		{\mathcal A}_n^{k+1}={\mathcal A}_n^k\cdot {\mathcal A}_n=\\=\begin{pmatrix}
			a_1^k & -a_1\cdot \mathscr{H}_{k-1}(a_1,a_2) & a_1\cdot a_2\cdot \mathscr{H}_{k-2}(a_1,a_2,a_3)          & \cdots & 0 \\
			0 & a_2^k & -a_2\cdot \mathscr{H}_{k-1}(a_2,a_3)     & \cdots & 0 \\
			0 & 0 & a_3^k     & \cdots & 0 \\
			\vdots & \vdots & \vdots    & \ddots & \vdots  \\
			0 & 0      & 0&  \cdots & a_n^k 
		\end{pmatrix} \cdot\begin{pmatrix}
			a_1 & -a_1 & 0          & \cdots & 0 \\
			0 & a_2 & -a_2    & \cdots & 0 \\
			0 & 0 & a_3     & \cdots & 0 \\
			\vdots & \vdots & \vdots    & \ddots & \vdots  \\
			0 & 0      & 0&  \cdots & a_n 
		\end{pmatrix} =\\
		=\begin{pmatrix}
			a_1^{k+1} & -a_1\cdot \mathscr{H}_{k}(a_1,a_2) & a_1\cdot a_2\cdot \mathscr{H}_{k-1}(a_1,a_2,a_3)          & \cdots & 0 \\
			0 & a_2^{k+1} & -a_2\cdot \mathscr{H}_{k}(a_2,a_3)     & \cdots & 0 \\
			0 & 0 & a_3^{k+1}     & \cdots & 0 \\
			\vdots & \vdots & \vdots    & \ddots & \vdots  \\
			0 & 0      & 0&  \cdots & a_n^{k+1} 
		\end{pmatrix}, 
	\end{multline*}
	which is what was required to be proven. 
	The proof for $k\geq n$ is similar. 
\end{proof}

Using \eqref{tai0} we obtain another form of the power ${\mathcal A}_n^k$.

\begin{corollary} 
	The power $k{\,\in\,}\mathbb{N}\cup\{0\}$ of matrix ${\mathcal A}_n$ is
	\begin{equation}\label{PowMat01}
		{\mathcal A}_n^k=(p_{sm})_{m=1,s=1}^{n,n},
	\end{equation}
	where
	\begin{equation}\label{PowMatEl02}
		p_{sm}=\left\{ \begin{array}{ll}
			0 & \mbox{if $m<s$};\\
			\prod\limits_{i=s}^{m-1}a_i\cdot\sum\limits_{j=s}^m\frac{ a_j^k}{\prod\limits_{i=s,j\neq i}^{m}(a_i-a_j)} & \mbox{if  $s< m$};\\
			a_m^k & \mbox{if $m=s$}.	
		\end{array} \right.
	\end{equation}
\end{corollary}

\begin{remark}
	Taking into account \eqref{Binom} we get that the power $k$ of the matrix \eqref{ed} for $k\geq n$ is
	matched with \eqref{FracPow00} but was obtained by different way.
\end{remark}

\section{Balakrishnan's approach to the power of an operator}

In this section, we consider Balakrishnan's approach and provide an example demonstrating its feasibility in the  matrix case.

Let $\{T_t\}_{t\geq 0}$ be a contraction semigroup of type $C_0$ (meaning it is a strongly continuous semigroup) on a Banach space $X$ with infinitesimal generator $(A,D(A))$. 

The fractional power $(-A)^\alpha$, $0<\alpha<1$ was defined by the formula (see \cite{Balakrishnan1959,Westphal1970,Yosida1980})
\begin{equation}\label{BF01}
	(-A)^\alpha f=\frac{1}{\Gamma(-\alpha)}\int\limits_0^\infty t^{-\alpha-1}(T_t-I)f(x)dt, \qquad f\in D(A).
\end{equation}
In the case $1<\alpha<\ell$, $\ell=2,3,...$
this formula can be written with the usage of "finite differences"\, $(I-T_t)^\ell$:
\begin{equation}\label{BF02}
	(-A)^\alpha f=\frac{1}{C_\alpha(\ell)}\int\limits_0^\infty t^{-\alpha-1}(I-T_t)^\ell f(x)dt,
\end{equation}
where  $C_\alpha(\ell)=\Gamma(-\alpha)A_\alpha(\ell)$,
$A_\alpha(\ell)=\sum\limits_{k=0}^\ell(-1)^{k-1}\binom {\ell}{k}=\sum\limits_{k=0}^\ell(-1)^{k-1}\frac {\ell!}{k!(\ell-k)!}.$

The negative power of the operator $(-A)$ for $0<\alpha<1$ can be defined by the equation
\begin{equation}\label{BF03}
	(-A)^{-\alpha} f=\frac{1}{\Gamma(\alpha)}\int\limits_0^\infty t^{\alpha-1}T_tf(x)dt.
\end{equation}
In order to get fractional integral of the order $\alpha$ greater than $1$ one can just apply iterated integral to \eqref{BF03}.

\begin{example}
Assume  $(n\times n)$ matrix $ \mathcal{A}_n$ of the form \eqref{ed}.
We can construct a contraction semigroup of type $C_0$  for the operator $(-\mathcal{A}_n)$ by formula $T_t^n=e^{-t\mathcal{A}_n}=\sum\limits_{k=0}^\infty\frac{(-t)^k}{k!}\mathcal{A}^k_n$. So, using \eqref{PowMat}, we obtain
$$
T_t^n=e^{-t\mathcal{A}_n}=\sum\limits_{k=0}^\infty\frac{(-t)^k}{k!}\mathcal{A}^k_n=$$
$$
=\sum\limits_{k=0}^\infty\frac{(-t)^k}{h^kk!}  \begin{pmatrix}
	1         & -k &     	\frac{k(k-1)}{2}   & \cdots & (-1)^{n-2}\frac{k(k-1)...(k-n+3)}{(n-2)!} & (-1)^{n-1}\frac{k(k-1)...(k-n+2)}{(n-1)!}\\
	0 &      1         & -k &     	 \cdots & (-1)^{n-3}\frac{k(k-1)...(k-n+4)}{(n-3)!} & (-1)^{n-2}\frac{k(k-1)...(k-n+3)}{(n-2)!}\\
	\vdots   &  \vdots & \vdots & \cdots & \vdots & \vdots\\
	0   & 0   & 0&\cdots & 1 & -k\\
	0   & 0   & 0 &  \cdots & 0 & 1
\end{pmatrix}=
$$
$$
=\begin{pmatrix}
	e^{-\frac{t}{h}}         & 	\frac{t}{h}e^{-\frac{t}{h}} &     \frac{1}{2}	\left( \frac{t}{h}\right)^2 e^{-\frac{t}{h}}   & \cdots &  	 \frac{1}{(n-2)!}\left( \frac{t}{h}\right)^{n-2}	e^{-\frac{t}{h}}  & \frac{1}{(n-1)!}\left( \frac{t}{h}\right)^{n-1}	e^{-\frac{t}{h}}\\
	0 & e^{-\frac{t}{h}}         & 	\frac{t}{h}e^{-\frac{t}{h}} &     	 \cdots &  	 	\frac{1}{(n-3)!}\left( \frac{t}{h}\right)^{n-3}	e^{-\frac{t}{h}} & \frac{1}{(n-2)!}\left( \frac{t}{h}\right)^{n-2}	e^{-\frac{t}{h}}\\
	\vdots   &  \vdots & \vdots & \cdots   & \vdots & \vdots\\
	0   & 0   & 0&\cdots & e^{-\frac{t}{h}} & 	\frac{t}{h}e^{-\frac{t}{h}}\\
	0   & 0   & 0 &  \cdots & 0 & e^{-\frac{t}{h}}
\end{pmatrix}.
$$

By formula \eqref{BF01} for $\alpha\in(0,1)$ we get
$$
\mathcal{A}_n^\alpha f=\frac{1}{\Gamma(-\alpha)}\int\limits_0^\infty t^{-\alpha-1}(e^{-t\mathcal{A}_n}-I)f(x)dt.
$$

Since
$$
\frac{1}{\Gamma(-\alpha)}\int\limits_0^\infty t^{-\alpha-1}(e^{-\frac{t}{h}}-1 )dt=\frac{1}{h^\alpha},
$$
and for $p+1\leq n$
$$
\frac{1}{\Gamma(-\alpha)}\frac{1}{h^{n-p}(n-p)!}  
\int\limits_0^\infty t^{-\alpha-1+n-p} e^{-\frac{t}{h}}dt=
$$
$$
= \frac{\Gamma (n-p-\alpha)}{h^\alpha\Gamma(-\alpha)(n-p)!}=\frac{(-1)^{n-p}\alpha(\alpha-1)...(\alpha-(n-p))}{h^\alpha(n-p)!},
$$
we get
$$
\mathcal{A}^\alpha_n=\frac{1}{h^\alpha}\begin{pmatrix}
	1         & -\alpha &     	\frac{\alpha(\alpha-1)}{2}   & \cdots & (-1)^{n-2}\frac{\alpha(\alpha-1)...(\alpha-n+3)}{(n-2)!} & (-1)^{n-1}\frac{\alpha(\alpha-1)...(\alpha-n+2)}{(n-1)!}\\
	0 &      1         & -\alpha &     	 \cdots & (-1)^{n-3}\frac{\alpha(\alpha-1)...(\alpha-n+4)}{(n-3)!} & (-1)^{n-2}\frac{\alpha(\alpha-1)...(\alpha-n+3)}{(n-2)!}\\
	\vdots   &  \vdots & \vdots & \cdots & \vdots & \vdots\\
	0   & 0   & 0&\cdots & 1 & -\alpha\\
	0   & 0   & 0 &  \cdots & 0 & 1
\end{pmatrix}.
$$
This formula  coincides with \eqref{FracPow0}.  
Thus, the same formula is derived in two different ways: by raising the matrix to a fractional power and by applying Balakrishnan's formula. 
\end{example}

\section{Rate of convergence: general approach}

In this section, we derive estimates for the rate of convergence of linear positive operators, given the semigroup's estimate. Understanding the rate of convergence of a sequence of operators is essential for both the underlying theory and its applications in fractional calculus. In \cite{KolokoltsovBook}, a comprehensive theorem on the convergence of Markov semigroups and processes was established (see p. 352, Theorem 8.1.1). Here, we derive a corollary from this theorem that applies to fractional powers of operators, including the rate of convergence.

Let $\Omega$   be a locally compact, separable metric space. Here, we only meet  $\Omega=\mathbb{R}^n$ or its subsets equipped with the Euclidean norm.

$C_b(\Omega)$ is the Banach space of 
bounded continuous functions on $\Omega$ equipped with the sup-norm $\|f\|=\sup\limits_{x\in \Omega}|f(x)|$.

$C_\infty(\Omega){\,\subset\,}C_b(\Omega)$  consists of $f$ such that $\lim\limits_{x\to\infty}f(x){\,=\,}0$, i.e. for all $\varepsilon$ there exists a compact
set $K{\,:\,}\sup\limits_{x\notin K}|f(x)|{\,<\,}\varepsilon$.

A semigroup of positive contraction linear operator $\{T_t\}_{t\geq 0}$  on $C_\infty(\Omega)$ 
is called {\bf Feller semigroup} if it satisfies the
following regularity conditions:
\begin{itemize}
	\item $T_tC_\infty(\Omega)\subset C_\infty(\Omega)$, $t\geq 0$;
	\item $T_tf(x)\to f(x)$, $t\downarrow0$, $\forall f\in C_\infty(\Omega)$.
\end{itemize}    
That gives strong continuity $\|T_t f-f\|\to 0$, $t\downarrow 0$, $\forall f\in C_\infty(\Omega)$.

It is well known that  each $C_0$-semigroup $T_t$ has the following estimate (see \cite{Yosida1980})
\begin{equation}\label{SemOts}
	\|T_t\|\leq M e^{mt},\qquad 0\leq t<\infty,
\end{equation}
where $M\in\mathbb{R}$, $M\geq 1$.

For any $h>0$, let $\Omega_h\subset\Omega$  be a closed subset. In
basic examples $\Omega=\mathbb{R}^n$ and $\Omega_h$ either coincides with $\Omega$ or is a lattice in $\mathbb{R}^n$ with
steps in each direction depending on $h$. By $\pi_h$ we shall denote here the projection
$C_\infty(\Omega)\mapsto C_\infty(\Omega_h)$ obtained by restriction.

Here we consider semigroup $T_t$ with its generator $A$ and
the sequence $T_t^h$  with   generators $A_h$.
We show that if $A_h\pi_h$ converges to $\pi_hA$ as $h\to 0$ then
$T_t^h\pi_h$ converges to $\pi_hT_t$  and
$(-A_h)^\alpha \pi_h$
converges to $\pi_h(-A)^\alpha$ as $h\to 0$, where fractional powers are given in the form of \eqref{BF01} for
$\alpha\in(0,1)$.

Next we will consider $\alpha\in(0,1)$.

\begin{theorem}\label{T2}
	Let $\{T_t\}_{t\geq 0}$ be a Feller semigroup
	in $C_\infty(\Omega)$  with generator $A$ and $B$ be a dense invariant (under $T_t$) 
	subspace of the domain of $A$ such that $B$ is itself a normed space under a norm $\|\cdot\|_B$ and
	\begin{equation}\label{Tt}
		\|T_tf\|_B\leq c(t)\|f\|_B
	\end{equation}
	for all $f$, for some non-decreasing continuous    function $c(t)$ on $\mathbb{R}_+$.
	
	Let $T^h_t$ be a Feller semigroup in $C_\infty(\Omega_h)$  with generator $A_h$ having domain
	containing $\pi_hB$ such that	
	\begin{equation}\label{Ah}
		\|(A_h\pi_h-\pi_h A)f\|\leq \omega(h)\|f\|_B
	\end{equation}
	with $\omega(h)=o(1)$   as $h\to 0$ (i.e. $\omega(h)\to 0$ as $h\to 0$), $\omega(h)\in(0,1)$. 
	Then 
	\begin{enumerate}
		\item for any $t>0$
		and $f\in C_\infty(\Omega)$ the sequence $T_s^hf$
		converges to $T_sf$ as $h\to 0$ uniformly for
		$s\in[0,t]$, in the sense that
		\begin{equation}\label{supTh}
			\sup\limits_{s\leq t}||(T_s^h\pi_h-\pi_hT_s)f||\to 0,\qquad h\to 0
		\end{equation}	
		(in case $\Omega_h=\Omega$  this is the usual convergence in $C_\infty(\Omega)$, 
		\item if
		$f\in B$, then 
		$$
		\sup\limits_{s\leq t}||(T_s^h\pi_h-  \pi_hT_s)f||\leq \omega(h)\|f\|_B\int\limits_0^tc(s)ds,
		$$
		\item  for  $m\leq 0$ in \eqref{SemOts}
		$$
		\|((-{A}_h)^\alpha\pi_h-\pi_h(-{A})^\alpha)f\|
		\leq C_1\|f\|_B \omega^\alpha(h),
		$$
		where $C_1=C_1(M,\alpha)$ is  some constant.
		\item  	 for  $m>0$ in \eqref{SemOts}
		$$
		\|((-{A}_h)^\alpha\pi_h-\pi_h(-{A})^\alpha)f\|
		\leq C_2\|f\|_B \left(\ln\frac{1}{\omega(h)}\right)^{-\alpha},
		$$
		where $C_2=C_2(M,m,\alpha)$ is  some constant.
		\item   if $\|T_t\|\leq M (1+t^p)$,  $p>0$
		$$
		\|((-{A}_h)^\alpha\pi_h-\pi_h(-{A})^\alpha)f\|
		\leq C_3\omega^{\frac{\alpha}{p+1}}(h)\|f\|_B ,
		$$
		where  $C_3=C_3(M,p,\alpha)$ is some constant.
	\end{enumerate}
\end{theorem}
\begin{proof}
	Let $f\in C_\infty(\Omega)$, then
	$$
	(  T_t^h\pi_h-  \pi_hT_t)f=(e^{-t {A}_h}\pi_h-\pi_he^{-t {A}})f=-e^{-(t-s) {A}_h}\pi_h e^{-s {A}}f|_{s=0}^t=
	$$
	$$
	=-\int\limits_0^t \frac{d}{ds}e^{-(t-s) {A}_h}\pi_h e^{-s {A}^{-1}}f ds=\int\limits_0^t  e^{-(t-s) {A}_h} ( {A}_h\pi_h-\pi_h{A})e^{-s{A}}f ds
	$$
	and
	\begin{multline}\label{En01}
		\|(  T_t^h\pi_h-  \pi_hT_t)f\| \leq \int\limits_0^t  \|( {A}_h\pi_h-\pi_h{A})e^{-s{A}}f\|  ds\leq  \int\limits_0^t\omega(h)\|e^{-s{A}}f\|_B ds=\\
		=\int\limits_0^t\omega(h)\|T_sf\|_B ds  \leq  \omega(h)\|f\|_B\int\limits_0^tc(s)ds.
	\end{multline}
	
	Therefore, by \eqref{BF01} for $\alpha\in(0,1)$  we get
	\begin{multline*}
		\|((-{A}_h)^\alpha\pi_h-\pi_h(-{A})^\alpha)f\|\leq \left|\frac{1}{\Gamma(-\alpha)}\right|\int\limits_0^\infty t^{-\alpha-1}\|(T_t^h\pi_h- \pi_hT_t)f\|dt =\\
		=\left|\frac{1}{\Gamma(-\alpha)}\right|\left( \int\limits_0^{K(\omega)} t^{-\alpha-1}\|(T_t^h\pi_h- \pi_hT_t)f\|dt+\int\limits_{K(\omega)}^\infty t^{-\alpha-1}\|(T_t^h\pi_h- \pi_hT_t)f\|dt \right), \\	
	\end{multline*}
	where $K(\omega)=K(\omega(h))$ is some positive function of $\omega(h)$.
	Applying inequality \eqref{En01} to the first integral by $(0,K(\omega))$ and taking into account strong continuity of $T_t$ we obtain
	\begin{multline}\label{Ner}
		\|((-{A}_h)^\alpha\pi_h-\pi_h(-{A})^\alpha)f\|\leq\\
		\leq \left|\frac{1}{\Gamma(-\alpha)}\right|\|f\|_B \left( \omega(h) \left|\int\limits_0^{K(\omega)} t^{-\alpha-1}\left( \int\limits_0^tc(s)ds\right) dt\right|+ 2\left|\,\int\limits_{K(\omega)}^\infty t^{-\alpha-1}dt\right|\right)=\\
		=\left|\frac{1}{\Gamma(-\alpha)}\right|\|f\|_B \left( \omega(h) \left|\int\limits_0^{K(\omega)} t^{-\alpha-1}\left( \int\limits_0^tc(s)ds\right) dt\right|+ \frac{2}{\alpha}K^{-\alpha}(\omega) \right).
	\end{multline}

	Taking into account \eqref{SemOts} and \eqref{Tt} we get $0\leq c(t)\leq Me^{mt}$, $M>0$, $m<\infty$
	and
	$$
	\int\limits_0^tc(s)ds\leq M\int\limits_0^t e^{ms}ds=\frac{M}{m}(e^{mt}-1)\leq Mte^{mt},\qquad t>0.
	$$
	Let first $m\leq 0$, then
	$$
	\int\limits_0^tc(s)ds\leq  Mt,\qquad t>0
	$$
	and \eqref{Ner} gives
	\begin{multline*}
		\|((-{A}_h)^\alpha\pi_h-\pi_h(-{A})^\alpha)f\|
		\leq  \left|\frac{1}{\Gamma(-\alpha)}\right|\|f\|_B \left(M \omega(h) \int\limits_0^{K(\omega)} t^{-\alpha}dt+ \frac{2}{\alpha}K^{-\alpha}(\omega) \right)=\\
		=\left|\frac{1}{\Gamma(-\alpha)}\right|\|f\|_B  \left(\frac{M}{1-\alpha}\omega(h)K^{1-\alpha}(\omega)+ \frac{2}{\alpha}K^{-\alpha}(\omega) \right).
	\end{multline*}	
	If we take $K(\omega)=\frac{1}{\omega(h)}$ we can write
	$$
	\|((-{A}_h)^\alpha\pi_h-\pi_h(-{A})^\alpha)f\|
	\leq \left|\frac{1}{\Gamma(-\alpha)}\right|  \left(\frac{M}{1-\alpha} + \frac{2}{\alpha}  \right)\|f\|_B \omega^\alpha(h)=C_1(M,\alpha)\|f\|_B\omega^\alpha(h).
	$$

	Now, when $m>0$ we obtain
	\begin{multline*}
		\|((-{A}_h)^\alpha\pi_h-\pi_h(-{A})^\alpha)f\|
		\leq  \left|\frac{1}{\Gamma(-\alpha)}\right|\|f\|_B \left(M \omega(h) \left|\int\limits_0^{K(\omega)} t^{-\alpha}e^{mt} dt\right|+ \frac{2}{\alpha}K^{-\alpha}(\omega) \right)\leq\\
		\leq\left|\frac{1}{\Gamma(-\alpha)}\right|\|f\|_B  \left(\frac{M}{1-\alpha}\omega(h)K^{1-\alpha}(\omega)e^{mK(\omega)}+ \frac{2}{\alpha}K^{-\alpha}(\omega) \right).
	\end{multline*}	
	Let $0<\omega(h)<1$. When $K(\omega)=\frac{1}{2m}\ln\frac{1}{\omega(h)}>0$ we get	
	\begin{multline*}
		\|((-{A}_h)^\alpha\pi_h-\pi_h(-{A})^\alpha)f\|\leq\\
		\leq \left|\frac{1}{\Gamma(-\alpha)}\right|\|f\|_B  \left(\frac{M}{1-\alpha}\omega(h)\left(\frac{1}{2m}\ln\frac{1}{\omega(h)}\right)^{1-\alpha} \frac{1}{\sqrt{\omega(h)}}+ \frac{2}{\alpha}\left(\frac{1}{2m}\ln\frac{1}{\omega(h)}\right)^{-\alpha}\right)=\\
		=\left|\frac{1}{\Gamma(-\alpha)}\right|\|f\|_B \left(\frac{1}{2m}\ln\frac{1}{\omega(h)}\right)^{-\alpha} \left(\frac{M}{2m(1-\alpha)}\sqrt{\omega(h)}\ln\frac{1}{\omega(h)}+ \frac{2}{\alpha}\right)\leq\\
		\leq\left|\frac{1}{\Gamma(-\alpha)}\right|\|f\|_B \left(\frac{1}{2m}\ln\frac{1}{\omega(h)}\right)^{-\alpha} \left(\frac{M}{2m(1-\alpha)}+ \frac{2}{\alpha}\right)\leq C_2(M,m,\alpha)\|f\|_B \left(\ln\frac{1}{\omega(h)}\right)^{-\alpha}.
	\end{multline*}
	
	Finally, we consider the case $\|T_t\|\leq M(1+t^p)$,  $p>0$. Then $0\leq c(t)\leq M(1+t^p)$, 
	$$
	\int\limits_0^tc(s)ds\leq M\int\limits_0^t (1+s^p)ds=M\left(t+\frac{t^{p+1}}{p+1}\right) ,\qquad t>0.
	$$
	and  \eqref{Ner} gives
	\begin{multline*}
		\|((-{A}_h)^\alpha\pi_h-\pi_h(-{A})^\alpha)f\|\leq \\
		\leq \left|\frac{1}{\Gamma(-\alpha)}\right|\|f\|_B \left( \omega(h) \left|\int\limits_0^{K(\omega)} t^{-\alpha-1}\left( \int\limits_0^tc(s)ds\right) dt\right|+ \frac{2}{\alpha}K^{-\alpha}(\omega) \right)\leq\\
		\leq \left|\frac{1}{\Gamma(-\alpha)}\right|\|f\|_B \left( M\omega(h) \left|\int\limits_0^{K(\omega)}\left( t^{-\alpha}+\frac{t^{p-\alpha}}{p+1}\right)dt\right|+ \frac{2}{\alpha}K^{-\alpha}(\omega) \right)=\\
		=\left|\frac{1}{\Gamma(-\alpha)}\right|\|f\|_B\left[M\omega(h)\left(\frac{K^{1-\alpha}(\omega)}{1-\alpha}+\frac{K^{p-\alpha+1}(\omega)}{(p+1)(p-\alpha+1)}\right)+ \frac{2}{\alpha}K^{-\alpha}(\omega)\right].
	\end{multline*}
	We take $K(\omega)=\omega^{-\frac{1}{p+1}}$, $\omega=\omega(h)$,  then 
	\begin{multline*}
		\|((-{A}_h)^\alpha\pi_h-\pi_h(-{A})^\alpha)f\|\leq \\
		\leq \left|\frac{1}{\Gamma(-\alpha)}\right|\|f\|_B \omega^{\frac{\alpha}{1+p}}\left[M \left(\frac{ \omega^{\frac{p}{p+1}}}{1-\alpha}+\frac{1}{(p+1)(p-\alpha+1)}\right)+ \frac{2}{\alpha}\right]\leq\\
		\leq C_3\omega^{\frac{\alpha}{1+p}}|\|f\|_B ,
	\end{multline*}
	where $C_3=C_3(M,p,\alpha)$.
\end{proof}

The theorem for the fractional order integral is  obtained by the similar method.

\begin{theorem}\label{T3}
		Let $\{T_t\}_{t\geq 0}$ be a Feller semigroup
		in $C_\infty(\Omega)$  with generator $A$ and $B$ be a dense invariant (under $T_t$) 
		subspace of the domain of $A$ such that $B$ is itself a normed space under a norm $\|\cdot\|_B$ and
		\begin{equation}\label{Tt3}
			\|T_tf\|_B\leq c(t)\|f\|_B,\qquad \varepsilon>0 
		\end{equation}
		for all $f$, for some non-decreasing continuous    function $c(t)$ on $\mathbb{R}_+$.
		Also, let there exits $\tau$ s.t. $T_tf(x)=0$ for all $f$, $x$ and $t\ge \tau$.
		
		Let $T^h_t$ be a Feller semigroup in $C_\infty(\Omega_h)$  with generator $A_h$ having domain
		containing $\pi_hB$ such that	
		\begin{equation}\label{Ah3}
			\|(A_h\pi_h-\pi_h A)f\|\leq \omega(h)\|f\|_B
		\end{equation}
		with $\omega(h)=o(1)$   as $h\to 0$ (i.e. $\omega(h)\to 0$ as $h\to 0$), $\omega(h)\in(0,1)$. 
					Then 
		the negative power of the operators $(-A)$ and $(-{A}_h)$ for $0<\alpha<1$ can be defined by   equations, respectively
		\begin{equation}\label{BFT3}
			(-A)^{-\alpha} f=\frac{1}{\Gamma(\alpha)}\int\limits_0^\infty t^{\alpha-1}T_tf(x)dt,\qquad (-{A}_h)^{-\alpha} f=\frac{1}{\Gamma(\alpha)}\int\limits_0^\infty t^{\alpha-1}T_t^hf(x)dt,
		\end{equation}
			and	
		$$
		 		\|((-A_h)^{-\alpha}\pi_h -\pi_h (-A)^{-\alpha})f \| \le C \omega(h) \|f\|_B\int\limits_0^{\tau}c(s)ds,\qquad C=C(\alpha,\tau).   
		$$	
\end{theorem}
\begin{proof}
	By \eqref{BF03} for $\alpha\in(0,1)$  we obtain
	\begin{multline*}
		\|((-{A}_h)^{-\alpha}\pi_h-\pi_h(-{A})^{-\alpha})f\|\leq  \frac{1}{\Gamma(\alpha)}\int\limits_0^\infty t^{\alpha-1}\|(T_t^h\pi_h- \pi_hT_t)f\|dt \leq\\
			\leq  \frac{1}{\Gamma(\alpha)}\omega(h)\|f\|_B\int\limits_0^\tau t^{\alpha-1} \left( \int\limits_0^tc(s)ds \right)dt=\frac{1}{\Gamma(\alpha)}\omega(h)\|f\|_B\int\limits_0^\tau c(s)\left( \int\limits_s^\tau t^{\alpha-1} dt \right)ds=\\
		=\frac{1}{\Gamma(\alpha+1)}\omega(h)\|f\|_B\int\limits_0^\tau c(s)(\tau^{\alpha}-s^\alpha)ds\leq\\
		\leq \frac{\tau^\alpha}{\Gamma(\alpha+1)}\omega(h)\|f\|_B\int\limits_0^\tau c(s)ds=C\omega(h)\|f\|_B\int\limits_0^\tau c(s)ds,
	\end{multline*}
	where $C=C(\alpha,\tau)=\frac{\tau^\alpha}{\Gamma(\alpha+1)}$.
\end{proof}
			
\begin{remark}
	It may happen that the "killing time"\, $\tau$ in Theorem  \eqref{T3} depends on $x$. Then the estimate becomes nonuniform, namely 
$$
	|((-A_h)^{-\alpha}\pi_h -\pi_h (-A)^{-\alpha})f(x)| \le \omega(h) \|f\|_B\int\limits_0^{\tau(x)} c(s) ds.
$$
	
\end{remark}		 

\section{From  the derivative of a function $f$ with respect to a function $g$ to the matrix operator}

Now, we will consider the derivative of a function as the infinitesimal variation of that function in relation to another function. Thus, the derivative with respect to a function serves as a generalization of the traditional concept of a derivative, enabling us to compare the infinitesimal variation of one function with that of another.

Let $C([a,b])$   be the Banach space equipped with the  norm 
$$
\|f\|_{C([a,b])}= \max\limits_{x\in [a,b] }|f (x)|.
$$

On $[a,b]$ we  consider an operator    
\begin{equation}\label{Amin}
	\mathcal{D}_gf(x)=f_{g}'(x)=\frac{df}{dg} = \lim\limits_{h\to  0}\frac{f(x)-f(x-h)}{g(x)-g(x-h)}=\frac{f'}{g'}.
\end{equation}

Operator \eqref{Amin} is  a    derivative of a function $f$ with respect to a function $g$.  Such double-sided operators were studied in details in \cite{Alfonso}. In particular in \cite{Alfonso} Taylor theorem for derivatives of a function with respect to a function was obtained.
Namely, if $g$ is continuous and injective and $f$ is at least $(n+1)$ times differentiable with respect to $g$, then there exists $d$ in
the open interval of extreme points $x$ and $a$ such that
\begin{equation}\label{Taylor}
	f(x)=\sum\limits_{k=0}^m\frac{f^{(k)}_g(a)}{k!}(g(x)-g(a))^k+R^g_m(x,a),
\end{equation}
$$
R^g_m(x,a)=\frac{f_g^{(m+1)}(d)}{m!}(g(x)-g(a))(g(x)-g(d))^m.
$$

$$
f(x)-f(a)=f'_g(a)(g(x)-g(a))=f_g''(d)(g(x)-g(a))(g(x)-g(d)).
$$

In our study we consider $h>0$ in \eqref{Amin}, that is the  left-hand derivative of a function $f$ with respect to a function $g$:
\begin{equation}\label{Amin02}
	Af(x)=f_{g,\ell}'(x)=\lim\limits_{h\to  0}\frac{f(x)-f(x-h)}{g(x)-g(x-h)},\qquad h>0.
\end{equation}
Obviously, formulas that are valid for two-sided derivatives will also be valid for one-sided ones.

Let $P{\,=\,}(x_{0},x_{1},\ldots ,x_{n})$ be a partition of $[a,x]$, $a<x\leq b$ that is\\
$a{\,=\,}x_{0}<x_{1}<x_{2}<\dots <x_{n}{\,=\,}x$, $x_k-x_{k-1}{\,=\,}h{\,=\,}\frac{x-a}{n}>0$, so $x_0=a$, $x_1=x_0+h$, $x_2{\,=\,}x_0+2h$,..., $x_{n-1}{\,=\,}x_0+(n-1)h$, $x_{n}{\,=\,}x_0+nh{\,=\,}x$, $f(a)=0$, $g(a)=0$. 

We have  $f_0{\,=\,}f(x_n)$, $f_1{\,=\,}f(x_{n-1}){\,=\,}f(x_n{\,-\,}h)$, ...,\\ $f_n{\,=\,}f(x_0){\,=\,}f(x_n{\,-\,}nh){\,=\,}f(a){\,=\,}0$,
$g_0{\,=\,}g(x_n)$, $g_1{\,=\,}g(x_{n-1}){\,=\,}g(x_n{\,-\,}h)$, ...,\\ $g_n{\,=\,}g(x_0){\,=\,}g(x_n-nh){\,=\,}g(a){\,=\,}0$,
$\mathbf{f}_n'=(f_0,f_{1},f_{2},...,f_{n-1})$, $I_{n}=\underbrace{(1,0,...,0)}_{n}$.

We associate the operator \eqref{Amin02} with the $(n{\times}n)$ matrix
\begin{equation}\label{Aminus}
	\mathcal{A}_n=\begin{pmatrix}
		a_1 & -a_1& 0          & \cdots & 0 \\
		0 & a_2 & -a_2    & \cdots & 0 \\
		0 & 0 & a_3     & \cdots & 0 \\
		\vdots & \vdots & \vdots    & \ddots & \vdots  \\
		0 & 0      & 0&  \cdots & a_n 
	\end{pmatrix}=\begin{pmatrix}
		\frac{1}{g_0-g_1} & -\frac{1}{g_0-g_1} & 0          & \cdots & 0 \\
		0 & \frac{1}{g_1-g_2} & -\frac{1}{g_1-g_2}    & \cdots & 0 \\
		0 & 0 & \frac{1}{g_2-g_3}     & \cdots & 0 \\
		\vdots & \vdots & \vdots    & \ddots & \vdots  \\
		0 & 0      & 0&  \cdots & \frac{1}{g_{n-1}-g_n} 
	\end{pmatrix},
\end{equation}
where $a_{k+1}=\frac{1}{g_k-g_{k+1}}$, $k=0,1,...,n-1$
and
\begin{multline}\label{OpA}
	Af(x)=f_{g,\ell}'(x)=\lim\limits_{h\rightarrow
		0} (\mathcal{A}_n \cdot \mathbf{f}_n)\cdot I_{n}=\lim\limits_{h\rightarrow
		0}\begin{pmatrix}
		\frac{f_0-f_1}{g_0-g_1} \\
		\frac{f_1-f_2}{g_1-g_2} \\
		\vdots    \\
		\frac{f_{n-1}-f_n}{g_{n-1}-g_n}
	\end{pmatrix}\cdot I_n=\\
	=\lim\limits_{h\rightarrow
		0}\frac{f_0-f_1}{g_0-g_1}=	\lim\limits_{h\rightarrow
		0} \frac{f(x_n)-f(x_{n-1})}{g(x_n)-g(x_{n-1})}.
\end{multline}

For finite $a{\,<\,}b$ let us denote by $C([a,b])_{kill(a)}$ (resp. $C([a,b])_{kill(b)}$ or 
$C([a,b])_{kill(a,b)}$) the subspace of functions from  $C([a,b])$ vanishing at $a$ (resp. at $b$ or at $a$ and $b$)
Similar notations will be apply also when $a$ or $b$ is infinite. For instance, in these notations,
$C_{\infty}(\mathbb{R})=C((-\infty, \infty))_{kill(\pm \infty)}$.

Let us consider the semigroup $T_t=e^{-tA}$ generated by $(-A)$, where
$$
Af(x)=\frac{df}{dg}=(f'/g')(x)=\frac{1}{g'(x)}\frac{df}{dx},\qquad f=f(x),\qquad g=g(x),\qquad x\in[a,b]\subset\mathbb{R}.
$$
The following statement is  obvious and we omit the proof (for details see \cite{KolokoltsovShishkina}).

\begin{prop}\label{pr1}\normalfont
	Let $g(x)$ be a monotonic continuously differentiable function with $c_1\le |g'(x)|\le d_1$
	for all $x\in \mathbb{R}$ and some positive $a_1\le b_1$. Let the operators $T_t$ act on continuous functions by the formula 
	\begin{equation}\label{TSem}
		T_tf(x)=f(g^{-1}(g(x)+t))=f(X_x(t)),
	\end{equation}
	where $X_x(t)=g^{-1}(g(x)+t)$ solves the ODE $\dot X=1/g'(X)$ with the initial condition $x$.
	If $1/g$ is increasing,   operators \eqref{TSem} form strongly continuous semigroups in spaces 
	$C_{\infty}((-\infty,b])_{kill(b)}$ and $C([a,b])_{kill(b)}$ with any finite $a$ and any finite or infinite $b>a$.  
	If $1/g$ is decreasing, these operators form strongly continuous semigroups in spaces
	$C_{\infty}([b, \infty))_{kill(b)}$ and $C([b,a])_{kill(b)}$ with any finite $a$ and any finite or infinite $b<a$. 
	In both cases, the generator of these semigroups act as the first order operator $Af(x)=(f'/g')(x)$ on smooth functions $f$. 
\end{prop}


For our purposes we are interested in invariant domains of these semigroups, where the generators can be 
effectively approximated by finite difference operators. For definiteness, we formulate and prove the next result for increasing $g$.
The case of decreasing $g$ is analogous.

\begin{theorem}\label{T4} \normalfont 
	\begin{enumerate}
		\item[(i)]	Under assumption of the Proposition \ref{pr1} above with increasing $g$, the contraction semigroup $T_t$ 
		acting in spaces $C_{\infty}((-\infty,b])_{kill(b)}$ and $C([a,b])_{kill(b)}$ has invariant domains
		$$
		C^1_{\infty}((-\infty,b])_{kill(b)}=\{f\in C_{\infty}((-\infty,b])_{kill(b)}: f' \in C_{\infty}((-\infty,b])_{kill(b)}\}
		$$
		and
		$$
		C^1([a,b])_{kill(b)}=\{f\in C([a,b])_{kill(b)}: f'\in C([a,b])_{kill(b)}\},
		$$
		respectively.  
		
		\item[(ii)] Equipping these spaces with the natural Banach norms
		$$
		\|f\|_1=\|f\|+\|f'\|,
		$$
		it follows that the norms of the operators $T_t$ in these spaces are bounded:
		\begin{equation}
			\label{eqboundder1}
			\|T_t\|_1 \le d_1/c_1.
		\end{equation}
		
		\item[(iii)] If additionally $|g''(x)| \le d_2$ for all $x$ with a constant $d_2$, then the previous estimate can be improved to 
		\begin{equation}
			\label{eqboundder2}
			\|T_t\|_1 \le \min\left\{\frac{d_1}{c_1}, 1+ K_1t\right\},
		\end{equation}
		where $K_1$ is a constant depending on $c_1,d_1,d_2$, so that $T_t$ is a bounded quasi-contraction in spaces $C^1$.
		Moreover, the subspaces 
		$$
		C^2_{\infty}((-\infty,b])_{kill(b)}=\{f\in C_{\infty}((-\infty,b])_{kill(b)}: f',f'' \in C_{\infty}((-\infty,b])_{kill(b)}\}
		$$
		and
		$$
		C^2([a,b])_{kill(b)}=\{f\in C([a,b])_{kill(b)}: f',f''\in C([a,b])_{kill(b)}\},
		$$
		respectively, equipped with the norm $\|f\|_2=\|f\|+\|f'\|+\|f''\|$
		also form invariant domains for $T_t$, and $T_t$ is uniformly bounded in these spaces:
		\begin{equation}
			\label{eqboundder3}
			\|T_t\|_2 \le K_2,
		\end{equation}
		where $K_2$ is a constant depending on $c_1,d_1,d_2$.
	\end{enumerate}
\end{theorem}
\begin{proof} 
	
\begin{enumerate}
	\item[(i)] It follows essentially from (ii) and (iii).
	\item[(ii)] Let us differentiate the expression for $T_tf(x)$ with respect to $x$ 
	(assuming $f$ is smooth) to obtain
	$$
	(T_tf)'(x)=f'(g^{-1}(g(x)+t))\frac{g'(x)}{ g'(g^{-1}(g(x)+t))}.
	$$   
	Therefore,
	$$
	\|(T_tf)'\|\le \|f'\| \frac{d_1}{c_1},
	$$
	and \eqref{eqboundder1} follows.
	\item[(iii)]  Firstly we have by the mean-value theorem that
	$$
	g^{-1}(g(x)+t)=x+\frac{t}{g'(g^{-1}(g(x)+\theta))}
	$$
	with a $\theta \in [0,t]$. Therefore,
	$$
	|g'(g^{-1}(g(x)+t))-g'(x)| \le \frac{d_2}{c_1}t .
	$$
	Hence, for 
	$$
	t< \frac{c_1^2}{2d_2},
	$$
	we can write (using the inequality $(1-y)^{-1} \le 1+2y$ for $y\in [0,1/2]$) that 
	$$
	\|(T_tf)'\|\le \|f'\|\left(1+\frac{2d_2}{c_1^2}t\right) .
	$$
	Finally, differentiating the expression for $T_tf(x)$ once more we find that all terms in the resulting expression are bounded
	implying \eqref{eqboundder3}.
\end{enumerate}
\end{proof}

Next we consider some basic examples, where $g'$ is not uniformly bounded from above and below, but may increase or decrease
on the bounds of the domains. It turns out that in many case the growth of $T_t$ on the invariant domains of smooth functions 
is only polynomial (not exponential, as it must be generally). 

\begin{example}\label{Ex2}
Let $g(x)=-x^{\beta}, x\ge 0, \beta\neq 0$.
The solution for characteristics  is $X_x(t)=(t+x^{\beta})^{1/\beta}$, so it moves to the right (left) for $\beta>0$ ($\beta<0$).
By prime we denote derivatives with respect to $x$. We obtain
$$
X'_x(t)=x^{\beta-1}(t+x^{\beta})^{\frac{1}{\beta}-1} >0.
$$
and
$$
X''_x(t)=t(\beta-1)x^{\beta-2}(t+x^{\beta})^{\frac{1}{\beta}-2},
$$
$$
X'''_x(t)=t(\beta-1)x^{\beta-3}(t+x^{\beta})^{\frac{1}{\beta}-3}[t(\beta-2)-(1+\beta)x^{\beta}].
$$

{\bf Case 1.1.} If $\beta\ge 1$, then $X''_x(t)\ge 0$, so that $X'_x(t)$ increases. Thus
$$
\sup_{x\in[a,\infty)} |X'_x(t)|=\lim_{x\to \infty} |X'_x(t)|=1,
$$
so that $T_t$ is a contraction on $C^1[0, \infty)$. Now if additionally $\beta\le 2$, then $X'''_x(t)\le 0$, so that
$$
\sup_{x\in[a,\infty)} |X''_x(t)|= |X''_a(t)|=t(\beta-1)a^{\beta-2}(t+a^{\beta})^{\frac{1}{\beta}-2} \sim t^{\frac{1}{\beta}-1},
$$
which decreases with $t$, so that $T_t$ is uniformly bounded in $C^2[a, \infty)$ for any $a{\,>\,}0$.
If $\beta{\,>\,}2$, then
\begin{multline*}
	\sup_{x\in[a,\infty)} |X''_x(t)|= |X''_x(t)|_{x^{\beta}=t(\beta-2)/(1+\beta)}=\\
	=t(\beta-1)\left[t \left( 1+\frac{\beta-2}{1+\beta}\right)\right]^{\frac{1}{\beta}-2} \left[t\cdot\frac{\beta-2}{1+\beta}\right]^{(\beta-2)/\beta}
	\sim t^{-1/\beta},
\end{multline*}
so also decreasing in $t$. Thus $T_t$ is bounded in $C_2[a,\infty)$.

{\bf Case 1.2.} If $\beta\in (0,1)$, then $X'_x(t)$ is decreasing and
$$
\sup_{x\in[a,\infty)} |X'_x(t)|= |X'_a(t)|=a^{\beta-1}(t+a^{\beta})^{\frac{1}{\beta}-1}\sim t^{\frac{1}{\beta}-1},
$$
so it increases in $t$ as a power. Since $X'''_x(t)>0$,
$$
\sup_{x\in[a,\infty)} |X''_x(t)|= |X''_a(t)|=t(\beta-1)a^{\beta-2}(t+a^{\beta})^{\frac{1}{\beta}-2} \sim t^{\frac{1}{\beta}-1},
$$
the same increase as for the first derivative. Thus
$$
\|T_t\|_{C^2[a,\infty)} \le M \left(1+t^{\frac{1}{\beta}-1}\right).
$$

{\bf Case 1.3.} If $\beta<0$, then $X''_x(t)<0$ and
$$
\sup_{x\in[a,\infty)} |X'_x(t)|= |X'_a(t)|=a^{\beta-1}(t+a^{\beta})^{\frac{1}{\beta}-1}\sim t^{\frac{1}{\beta}-1},
$$
is decreasing in $t$, so that $T_t$ is bounded in $C^1[a,\infty)$.

If $\beta\in [-1,0)$, then $X'''_x(t)>0$, and
$$
\sup_{x\in[a,\infty)} |X''_x(t)|= |X''_a(t)|=t|\beta-1|a^{\beta-2}(t+a^{\beta})^{\frac{1}{\beta}-2} \sim t^{\frac{1}{\beta}-1},
$$
so that $T_t$ is bounded in $C^2[a,\infty)$.

Finally, if $\beta<-1$, then
\begin{multline*}
	\sup_{x\in[a,\infty)} |X''_x(t)|= |X''_x(t)|_{x^{\beta}=t(\beta-2)/(1+\beta)}=\\
	=t(\beta-1)\left[t\left(1+\frac{\beta-2}{1+\beta}\right)\right]^{\frac{1}{\beta}-2}\left [t\cdot \frac{\beta-2}{1+\beta}\right]^{1-\frac{2}{\beta}}
	\sim t^{-1/\beta},
\end{multline*}
Thus
$$
\|T_t\|_{C^2[a,\infty)} \le M \left(1+t^{-\frac{1}{\beta}}\right).
$$
\end{example}

\begin{example}\label{Ex3}
Let $g(x)=x^{\beta}, x\ge 0, \beta\neq 0$.
The solution for characteristics  is $X_x(t)=(-t+x^{\beta})^{1/\beta}$, so it moves to the right (left) for $\beta<0$ ($\beta>0$).
By prime we denote derivatives with respect to $x$. We have
$$
X'_x(t)=x^{\beta-1}(-t+x^{\beta})^{\frac{1}{\beta}-1} >0
$$
and
$$
X''_x(t)=-t(\beta-1)x^{\beta-2}(-t+x^{\beta})^{\frac{1}{\beta}-2},
$$
$$
X'''_x(t)=-t(\beta-1)x^{\beta-3}(-t+x^{\beta})^{\frac{1}{\beta}-3}[-t(\beta-2)-(1+\beta)x^{\beta}].
$$

If $\beta<0$, then the characteristics $X_x(t)$ reaches $\infty$ in time $t=x^{\beta}=x^{-|\beta|}$.
If $\beta>0$, then the characteristics $X_x(t)$ reaches $0$ in time $t=x^{\beta}$.
Leaving aside the first case choose $\beta>0$ and define
$$
X_x(t)=\left\{
\begin{aligned}
	(-t+x^{\beta})^{1/\beta},& \quad t\le x^{\beta}; \\
	0,& \quad t\ge x^{\beta}.
\end{aligned}
\right.
$$
The derivatives are
\begin{equation}
	\label{eqder1}
	X'_x(t)=\left\{
	\begin{aligned}
		x^{\beta-1}(-t+x^{\beta})^{\frac{1}{\beta}-1},& \quad t\le x^{\beta}, \\
		0,& \quad t\ge x^{\beta},
	\end{aligned}
	\right.
\end{equation}
\begin{equation}
	\label{eqder2}
	X''_x(t)=\left\{
	\begin{aligned}
		t(1-\beta)x^{\beta-2}(-t+x^{\beta})^{\frac{1}{\beta}-2},& \quad t\le x^{\beta}, \\
		0,& \quad t\ge x^{\beta}.
	\end{aligned}
	\right.
\end{equation}

Since $X''_x(t)>0$, we see that
$$
\sup_{x\in[0, a]} |X'_x(t)|= \sup_{x\in[t^{1/\beta}, a]} |X'_x(t)|=|X'_a(t)|
=a^{\beta-1}(-t+a^{\beta})^{\frac{1}{\beta}-1}\sim t^{\frac{1}{\beta}-1}.
$$

Assume now that $\beta\in (0,1)$. Then this sup decreases with $t$ and
$$
\sup_{t\in[0,x^{\beta}]} |X'_x(t)|= \lim_{t\to x^{\beta}} |X'_x(t)|=0,
$$
and \eqref{eqder1} defines a continuous function, and
$T_t$ is continuous in $C^1[0,a]$. Since
$$
\lim\limits_{a\to \infty} \sup_{x\in[0, a]} |X'_x(t)|= 1,
$$
we even have that $T_t$ is continuous in $C^1[0,\infty]$ and
$$
\|T_t\|_{C^1[0,\infty)}\le M(1+t^{\frac{1}{\beta}-1}).
$$
Next, if
$\beta <1/2$, then
$$
\lim_{t\to x^{\beta}} |X''_x(t)|=0,
$$
so that \eqref{eqder2} defines a continuous function,
$T_t$ is continuous in $C^2[0,a]$.
Since
\[
\lim_{a\to \infty} \sup_{x\in[0, a]} |X''_x(t)|= 0,
\]
we even have that $T_t$ is continuous in $C^2[0,\infty]$.
Moreover, since
$$
t\cdot\frac{2-\beta}{1+\beta}> t,
$$
we have that
$$
\sup_{x\in[t^{1/\beta},\infty)} |X''_x(t)|=|X''_x(t)|_{x^{\beta}=t(2-\beta)/(1+\beta)}
$$
$$
=t|\beta-1|\left[t\cdot\frac{|\beta|}{1+\beta}\right]^{\frac{1}{\beta}-2} \left[t\cdot\frac{2-\beta}{1+\beta}\right]^{1-\frac{2}{\beta}}
\sim t^{-1/\beta},
$$
which is decreasing in $t$. Thus
$$
\|T_t\|_{C^2[0,\infty)}\le M(1+t^{\frac{1}{\beta}-1}).
$$

If $\beta \in [1/2,1]$, then
$$
t\cdot\frac{2-\beta}{1+\beta}\le t,
$$
and
$$
\sup_{x\in[t^{1/\beta},\infty)} |X''_x(t)|= |X''_x(t)|_{x^{\beta}=t}=\infty.
$$
Thus $T_t$ is not continuous in $C^2[0,\infty)$.
\end{example}

\section{Convergence of matrix operator: strictly increasing  measures}

It is clear that constructions of the Section \ref{Sec02} can be applied  to creating the fractional power of a  derivative of a function $f$ with respect to a function $g$. Here we consider the case of strictly increasing function $g(x)$.

Let us consider  an operator $A$ in the form \eqref{Amin02}:
$$
Af(x)=f_{g,\ell}'(x)=\lim\limits_{h\to  0}\frac{f(x)-f(x-h)}{g(x)-g(x-h)},\qquad h>0.
$$
Operator $(-A)$  is an infinitesimal generator of a semigroup $T_t{\,=\,}e^{-tA}$, i.e., we take  $e^{-tA}u(x):=w(x,t)$, where $w$ is the solution of the problem
\begin{equation}\label{Cauchy}
	\left\{ \begin{array}{ll}
		\partial_t w(x,t)+Aw(x,t)=0 & \mbox{if $x\in\mathbb{R},\qquad t>0$};\\
		w(x,0)=u(x) & \mbox{if  $x\in\mathbb{R}$}.\end{array} \right.
\end{equation}

To the operator $A$ we associate the matrix operator $\mathcal{A}_n$ defined by \eqref{Aminus}.
Since we can calculate the power $k{\,\in\,}\mathbb{N}\cup\{0\}$  of the matrix $\mathcal{A}_n$ by Theorem \ref{teopowder} we can construct the operator
\begin{equation}\label{sg}
	T_t^n=e^{-t\mathcal{A}_n}=\sum\limits_{k=0}^\infty\frac{(-t)^k}{k!}\mathcal{A}_n^k,\qquad 0\leq t,
\end{equation} 
which is a contraction semigroup of type $C_0$   when
$a_1{\,>\,}0,a_2{\,>\,}0,...,a_n{\,>\,}0$, where $a_{k+1}{\,=\,}\frac{1}{g_{k}{\,-\,}g_{k+1}}$, $k{\,=\,}0,1,{...},n{\,-\,}1$, i.e. when $g(x)$ is strictly increasing function on $[a,b]$.
That means that we can construct fractional power of $(-\mathcal{A}_n)^\alpha$ by formula \eqref{BF01} 
(or \eqref{BF02}, or \eqref{BF03}).
Subsequently, we can construct the fractional power $\alpha\in \mathbb{R}$ of an operator $(-A)$ as a limit by the formula
$$
(-A)^\alpha f=\lim\limits_{h\rightarrow
	0} (-\mathcal{A}_n)^\alpha\cdot \mathbf{f}_n\cdot I_{n}.
$$		
Let's obtain an estimate of the norm of the difference of operators $(-A)^\alpha$ and $(-\mathcal{A}_n)^\alpha$.

Next we will consider $\alpha\in(0,1)$.

\begin{theorem}
	Let $f{\,\in\,}C^2[a,b]$ be twice continuously differentiable from the left with respect to $g$ on $[a,b]$, $g{\,\in\,}C^2[a,b]$, $g'>0$  on $(a,b]$,   $f(a){\,=\,}g(a){\,=\,}0$. 
	Operator $A$ is defined by 
	$$
	Af(x)=f_{g,\ell}'(x)=\frac{f_{\ell}'(x)}{g_{\ell}'(x)}=\lim\limits_{h\to  0}\frac{f(x)-f(x-h)}{g(x)-g(x-h)},\qquad h>0
	$$
	and its fractional power $(-{A})^\alpha$ is defined  by formula \eqref{BF01} for $\alpha{\,\in\,}(0,1)$.
	
	Operator $(-A)$  is an infinitesimal generator of a contraction semigroup $T_t$ on $C([a,b])_{kill(a)}$.
	In the invariant domain $C^2([a,b])_{kill(a)}$ the $T_t$ is
	\begin{enumerate}
		\item[(i)] bounded in the assumptions of Theorem \ref{T4} or
		\item[(ii)] $\|T_t\|\leq M (1+t^p)$, $0\leq t<\infty$,	
		where $M,p\in\mathbb{R}$, $M\geq 1$, $p\geq0$ in more general examples  \ref{Ex2}, \ref{Ex3} above.
	\end{enumerate}
	
	On the segment $[a,x]$ a partition with a step $h{\,\in\,}(0,1)$ is given. For this  partition generates the matrix operator $\mathcal{A}_n$   by formula \eqref{Aminus}. Operator $(-\mathcal{A}_n)^\alpha$ is the fractional power of $(-\mathcal{A}_n)$ defined by by formula \eqref{BF01}.
	Then 	  
	\begin{enumerate}
		\item  under condition (i)
		\begin{equation}\label{es1}
			\|((-{A})^\alpha f(x_n)-((-\mathcal{A}_n)^\alpha \cdot\mathbf{f}_n)\cdot I_{n}\|_{C([a,b])}
			\leq C_1\|f\|_{C^2([a,b])} \omega^\alpha(h),
		\end{equation} 
		where $C_1=C_1(M,\alpha)$ is  some constant,
		\item  	 under condition (ii)
	\begin{equation}\label{es2}
		\|((-{A})^\alpha f(x_n)-((-\mathcal{A}_n)^\alpha \cdot\mathbf{f}_n)\cdot I_{n}\|_{C([a,b])}
		\leq C_2\omega^{\frac{\alpha}{p+1}}(h)\|f\|_{C^2([a,b])},
	\end{equation} 
		where   $C_2=C_2(M,p,\alpha)$  is  some constant.
	\end{enumerate}
\end{theorem}
\begin{proof}
	Since $f(x){\,\in\,}C^2[a,b]$ and $g(x){\,\in\,}C^2[a,b]$, $g'>0$   on $(a,b]$, then 
	the problem \eqref{Cauchy} is well-posed, $T_t$ is
	the contraction semigroup acting in spaces from Theorem \ref{T4}
and operator $T_t$ satisfies to inequality
	$$
	||T_t||\leq M (1+t^p),\qquad M,p\in\mathbb{R}, \qquad M\geq 1, \qquad p\geq 0.
	$$

	By \eqref{OpA} we obtain
	$$
	\|Af(x)-(\mathcal{A}_n \cdot \mathbf{f}_n)\cdot I_{n}\|_{C([a,b])}=\left\|f_{g-}'(x_n)-\frac{f(x_n)-f(x_{n-1})}{g(x_n)-g(x_{n-1})}\right\|_{C([a,b])}.
	$$
	Putting   $m=1$ in \eqref{Taylor}, we get 
	$$
	\left\|f_{g-}'(x_n)-\frac{f(x_n)-f(x_{n-1})}{g(x_n)-g(x_{n-1})}\right\|_{C([a,b])}\leq \max\limits_{x\in[x_{n-1},x_{n}]}|f_g''(x)g'(x)| \cdot h.
	$$
	Therefore,
	$$
	\|Af(x)-(\mathcal{A}_n \cdot \mathbf{f}_n)\cdot I_{n}\|_{C([a,b])}\leq \mathcal{C}_1\cdot h,
	$$
	where $\mathcal{C}_1=\|f_g''(x)g'(x)\|_{C([a,b])}\in\mathbb{R}$ that gives \eqref{Ah}.
	Then, applying Theorem \ref{T2} we obtain estimates	\eqref{es1} and \eqref{es2}.
\end{proof}

\section{Arbitrary power of two-band matrices}

In section \ref{NP} the natural power of  two-band matrices were obtained. Next, to obtain fractional powers  formulas  \eqref{BF01} and \eqref{BF03} 
should be used. In this section we obtain the direct formula for calculating  fractional powers   
of two-band matrices.

\begin{theorem}\label{theo01}	Let ${\rm det}\, {\mathcal A}_n\neq 0$, $a_1\neq a_2\neq ...\neq a_n$, then the matrix
	$$
	{\mathcal A}_n=\begin{pmatrix}
		a_1 & -a_1 & 0          & \cdots & 0 \\
		0 & a_2 & -a_2    & \cdots & 0 \\
		0 & 0 & a_3     & \cdots & 0 \\
		\vdots & \vdots & \vdots    & \ddots & \vdots  \\
		0 & 0      & 0&  \cdots & a_n 
	\end{pmatrix}
	$$
	is diagonalizable and eigenvalue decomposition for matrix $A$ is
	\begin{equation}\label{diag}
		{\mathcal A}_n=PDP^{-1},
	\end{equation}
	where	
	$P=(b_{s\,m})_{m=1,s=1}^{n,n}$, $P^{-1}=(c_{s\,m})_{m=1,s=1}^{n,n}$, $D=(d_{sm})_{m=1,s=1}^{n,n}$,
	$$
	b_{sm}= \left\{ \begin{array}{ll}
		1 & \mbox{if $s=m$};\\
		0 & \mbox{if $m<s$};\\
		\frac{ \prod\limits_{i=s}^{m-1}a_i}{\prod\limits_{i=s}^{m-1}(a_i-a_m)} & \mbox{if  $s<m$},
	\end{array} \right. 
	\qquad
	c_{sm}= \left\{ \begin{array}{ll}
		1 & \mbox{if $s=m$;}\\
		0 & \mbox{if $m<s$};\\
		\frac{\prod\limits_{i=s}^{m-1}a_i}{\prod\limits_{i={s+1}}^{m}(a_{i}-a_{s})} & \mbox{if   $s< m$},
	\end{array} \right.  
	$$
	$$
	d_{s\,m}= \left\{ \begin{array}{ll}
		a_m & \mbox{if $m=s$};\\
		0   & \mbox{if $s\neq m$}.\end{array} \right. 
	$$
\end{theorem}
\begin{proof}
	
	Eigenvalues for $A$ are $a_1,a_2,...,a_n$. Since $a_1\neq a_2\neq ...\neq a_n$ then matrix ${\mathcal A}_n$ is diagonalizable  and can be written in the form
	$ {\mathcal A}_n=PDP^{-1}$, where 
	$$
	D=\begin{pmatrix}
		a_{1}   & 0     &         \cdots & 0 \\
		0       & a_{2} &         \cdots & 0 \\
		\vdots  & \vdots& \ddots         & \vdots  \\
		0       & 0     &   \cdots       & a_{n} 
	\end{pmatrix}
	$$ is  a diagonal matrix and	
	$$
	P=\begin{pmatrix}
		b_{11} & b_{12}          & \cdots & b_{1n} \\
		0 & b_{22}      & \cdots & b_{2n} \\
		\vdots & \vdots      & \ddots & \vdots  \\
		0 & 0      & \cdots & b_{nn} 
	\end{pmatrix}
	$$ is an invertible matrix. Let find $P$ and $P^{-1}$.
	Since
	$$
	PD= \begin{pmatrix}
		b_{11} & b_{12}          & \cdots & b_{1n} \\
		0 & b_{22}      & \cdots & b_{2n} \\
		\vdots & \vdots      & \ddots & \vdots  \\
		0 & 0      & \cdots & b_{nn} 
	\end{pmatrix}\cdot \begin{pmatrix}
		a_{1}   & 0     &         \cdots & 0 \\
		0       & a_{2} &         \cdots & 0 \\
		\vdots  & \vdots& \ddots         & \vdots  \\
		0       & 0     &   \cdots       & a_{n} 
	\end{pmatrix}	=\begin{pmatrix}
		a_1\,	b_{11}  & a_2\,	b_{12}                     & \cdots & a_n\,	b_{1n}  \\
		0  & a_2\, b_{22}        & \cdots & a_n\,b_{2n} \\
		\vdots      & \vdots            & \ddots   & \vdots  \\
		0 & 0      & \cdots &a_n\, b_{nn} 
	\end{pmatrix},
	$$
	then an equality $AP=PD$ gives
	$$
	\begin{pmatrix}
		a_1 & -a_1 & 0          & \cdots & 0 \\
		0 & a_2 & -a_2    & \cdots & 0 \\
		0 & 0 & a_3     & \cdots & 0 \\
		\vdots & \vdots & \vdots    & \ddots & \vdots  \\
		0 & 0      & 0&  \cdots & a_n 
	\end{pmatrix}\cdot \begin{pmatrix}
		b_{11} & b_{12}  & b_{13}         & \cdots & b_{1n} \\
		0 & b_{22}   & b_{23}   & \cdots & b_{2n} \\
		0 & 0& b_{33}      & \cdots & b_{3n} \\
		\vdots & \vdots  &  \vdots   & \ddots & \vdots  \\
		0 & 0 & 0      & \cdots & b_{nn} 
	\end{pmatrix}
	=\begin{pmatrix}
		a_1\,	b_{11}  & a_2\,	b_{12}   &      a_3\,	b_{13}           & \cdots & a_n\,	b_{1n}  \\
		0  & a_2\, b_{22} & a_3\, b_{23}        & \cdots & a_n\,b_{2n} \\
		0  &  	0  & a_3\, b_{32}        & \cdots & a_n\,b_{3n} \\
		\vdots      & \vdots   & \vdots           & \ddots   & \vdots  \\
		0 & 0  & 0    & \cdots &a_n\, b_{nn} 
	\end{pmatrix}
	$$
	or
	$$
	\begin{pmatrix}
		a_1b_{11} & a_1(b_{12}-b_{22})  & a_1(b_{13}-b_{23})         & \cdots & a_1(b_{1n}-b_{2n}) \\
		0 &a_2 b_{22}   & a_2(b_{23}-b_{33})   & \cdots & a_2(b_{2n}-b_{3n}) \\
		0 & 0& a_3b_{33}      & \cdots & a_3(b_{3n}-b_{4n}) \\
		\vdots & \vdots  &  \vdots   & \ddots & \vdots  \\
		0 & 0 & 0      & \cdots & a_nb_{nn} 
	\end{pmatrix}
	=\begin{pmatrix}
		a_1\,	b_{11}  & a_2\,	b_{12}   &      a_3\,	b_{13}           & \cdots & a_n\,	b_{1n}  \\
		0  & a_2\, b_{22} & a_3\, b_{23}        & \cdots & a_n\,b_{2n} \\
		0  &  	0  & a_3\, b_{32}        & \cdots & a_n\,b_{3n} \\
		\vdots      & \vdots   & \vdots           & \ddots   & \vdots  \\
		0 & 0  & 0    & \cdots &a_n\, b_{nn} 
	\end{pmatrix}.
	$$
	For the first row we obtain
	$$
	\left\{ \begin{array}{ll}
		a_1\cdot b_{11}&=a_1\cdot b_{11}\\
		a_1(b_{12}-b_{22}) &=a_2\cdot b_{12}\\
		...\\
		a_1(b_{1n}-b_{2n})&=a_n\cdot b_{1n} 
	\end{array} \right.
	$$
	for the second row:
	$$
	\left\{ \begin{array}{ll}
		a_2\cdot b_{22}&=a_2\cdot b_{22}\\
		a_2(b_{23}-b_{33})&=a_3\cdot b_{23}\\
		...\\
		a_2(b_{2n}-b_{3n})&=a_n\cdot b_{2n} 
	\end{array} \right. ,...,
	$$
	etc. and for the n-th row
	$$
	a_n\cdot b_{nn}=a_n\cdot b_{nn}.
	$$
	Let $s$ be the number of row, $m$ be the number of column,  $s=1,...,n$, $m=1,..,n$. For all diagonal elements we put
	$$
	b_{11}=b_{22}=...=b_{nn}=1.
	$$
	If  $s< m$ then 
	we can write the system
	$$
	b_{sm}=\frac{a_s}{a_s-a_m}b_{s+1\, m}.
	$$
	For $s=m-1$ we get
	$$
	b_{m-1\,m}=\frac{a_{m-1}}{a_{m-1}-a_m},
	$$
	next, for $s=m-2$
	$$
	b_{m-2\,m}=\frac{a_{m-2}}{a_{m-2}-a_m}b_{m-1\, m}=\frac{a_{m-1}a_{m-2}}{(a_{m-1}-a_m)(a_{m-2}-a_m)}.
	$$
	Therefore, for $s<m$ 
	$$
	b_{m-s\,m}=\frac{\prod\limits_{i=1}^{s} a_i}{\prod\limits_{i=1}^{s}(a_{m-i}-a_m)}
	$$	
	or for $s<m$
	$$
	b_{sm}=\frac{ \prod\limits_{i=s}^{m-1}a_i}{\prod\limits_{i=s}^{m-1}(a_i-a_m)}.
	$$
	
	Now let find $P^{-1}$. Inverse of an  lower triangular matrix $P$ is another upper
	triangular matrix
	$$
	P^{-1}=\begin{pmatrix}
		c_{11} & c_{12}           & \cdots &    c_{1n} \\
		0 & c_{22}      & \cdots &   c_{2n} \\
		\vdots & \vdots      & \ddots &   \vdots  \\
		0 & 0      & \cdots &   c_{nn} 
	\end{pmatrix}.
	$$
	Let find  elements $c_{sm}$ of $P^{-1}$ from the system
	$$
	PP^{-1}= 	\begin{pmatrix}
		b_{11} & b_{12}          & \cdots & b_{1n} \\
		0 & b_{22}      & \cdots & b_{2n} \\
		\vdots & \vdots      & \ddots & \vdots  \\
		0 & 0      & \cdots & b_{nn} 
	\end{pmatrix}\cdot\begin{pmatrix}
		c_{11} & c_{12}           & \cdots &    c_{1n} \\
		0 & c_{22}      & \cdots &   c_{2n} \\
		\vdots & \vdots      & \ddots &   \vdots  \\
		0 & 0      & \cdots &   c_{nn} 
	\end{pmatrix}
	=\begin{pmatrix}
		1         & 0        &   \cdots & 0 \\
		0         & 1        &   \cdots & 0 \\
		\vdots    &   \vdots & \ddots   & \vdots  \\
		0         & 0        &   \cdots & 1 
	\end{pmatrix}.
	$$
	It is clear that
	$$
	b_{mm}c_{mm}=1,\qquad m=1,...,n
	$$
	so diagonal elements of $P^{-1}$ are reciprocal of the elements of $P$:
	$$
	c_{mm}=\frac{1}{b_{mm}}=1,\qquad m=1,...,n.
	$$
	For other elements we get
	$$
	\sum\limits_{i=s}^kb_{si}c_{ik}=0,\qquad s=1,...,n,\qquad k=s+1,...,n.
	$$
	Putting $k=s+1$ we obtain
	$$
	\sum\limits_{i=s}^{s+1}b_{si}c_{is+1}=b_{ss}c_{ss+1}+b_{ss+1}c_{s+1s+1}=0,\qquad s=1,...,n-1
	$$
	or
	$$
	c_{ss+1}=-\frac{b_{ss+1}}{b_{ss}}c_{s+1s+1}=-b_{ss+1}=\frac{a_{s}}{a_{s+1}-a_{s}},\qquad s=1,...,n-1.
	$$
	For $k=s+2$ we get	
	$$
	\sum\limits_{i=s}^{s+2}b_{si}c_{is+2}=b_{ss}c_{ss+2}+b_{ss+1}c_{s+1s+2}+b_{ss+2}c_{s+2s+2}=0,\qquad s=1,...,n-2,
	$$
	$$
	c_{ss+2}=-\frac{1}{b_{ss}}(b_{ss+1}c_{s+1s+2}+b_{ss+2}c_{s+2s+2})
	=-(b_{ss+1}c_{s+1s+2}+b_{ss+2})=
	$$
	$$
	=-\left( -\frac{a_{s}a_{s+1}}{(a_{s+1}-a_s)(a_{s+2}-a_{s+1})}+
	\frac{a_{s}a_{s+1}}{(a_{s+2}-a_{s})(a_{s+2}-a_{s+1})}\right) =
	$$
	$$
	=- \frac{a_sa_{s+1}(a_{s+1}-a_{s+2})}{(a_{s+1}-a_s)(a_{s+2}-a_{s})(a_{s+2}-a_{s+1})} =\frac{a_sa_{s+1}}{(a_{s+1}-a_{s})(a_{s+2}-a_{s})}.
	$$
	
	For $k=s+3$ we get	
	$$
	\sum\limits_{i=s}^{s+3}b_{si}c_{is+3}=b_{ss}c_{ss+3}+b_{ss+1}c_{s+1s+3}+b_{ss+2}c_{s+2s+3}+b_{ss+3}c_{s+3s+3}=0,\qquad s=1,...,n-3,
	$$
	$$
	c_{ss+3}=-\frac{1}{b_{ss}}(b_{ss+1}c_{s+1s+3}+b_{ss+2}c_{s+2s+3}+b_{ss+3}c_{s+3s+3})=
	$$
	$$=-(b_{ss+1}c_{s+1s+3}+b_{ss+2}c_{s+2s+3}+b_{ss+3})=
	$$
	$$
	=-\left(-\frac{a_sa_{s+1}a_{s+2}}{(a_{s+1}-a_{s})(a_{s+2}-a_{s+1})(a_{s+3}-a_{s+1})}+
	\frac{a_sa_{s+1}a_{s+2}}{(a_{s+2}-a_{s})(a_{s+2}-a_{s+1})(a_{s+3}-a_{s+2})}-\right.
	$$
	$$
	\left.-\frac{a_sa_{s+1}a_{s+2}}{(a_{s+3}-a_{s})(a_{s+3}-a_{s+1})(a_{s+3}-a_{s+2})} \right)= 
	$$
	$$
	=\frac{a_sa_{s+1}a_{s+2}}{(a_{s+1}-a_{s})(a_{s+2}-a_{s})(a_{s+3}-a_{s})}.
	$$

	Therefore for $s=1,...,n$,
	$$
	c_{ss+m}=\frac{\prod\limits_{i=s}^{s+m-1}a_i}{\prod\limits_{i={s+1}}^{s+m}(a_{i}-a_{s})}
	$$
	or for   for $s<m$
	$$
	c_{sm}=\frac{\prod\limits_{i=s}^{m-1}a_i}{\prod\limits_{i={s+1}}^{m}(a_{i}-a_{s})}.
	$$
\end{proof}

\begin{theorem}\label{teopow} 	Let ${\rm det}\, {\mathcal A}_n\neq 0$, $a_1\neq a_2\neq ...\neq a_n$, then the power $\alpha\in\mathbb{R}$ of matrix ${\mathcal A}_n$ is
	\begin{equation}\label{PowMat03}
		{\mathcal A}_n^\alpha=(p_{sm})_{m=1,s=1}^{n,n},
	\end{equation}
	where
	\begin{equation}\label{PowMatEl01}
		p_{sm}=\left\{ \begin{array}{ll}
			0 & \mbox{if $m<s$};\\
			\prod\limits_{i=s}^{m-1}a_i\cdot\sum\limits_{j=s}^m\frac{ a_j^\alpha}{\prod\limits_{i=s,j\neq i}^{m}(a_i-a_j)} & \mbox{if  $s< m$};\\
			a_m^\alpha & \mbox{if $m=s$}.	
		\end{array} \right.
	\end{equation}
\end{theorem}	
\begin{proof} By Theorem \ref{theo01} we get ${\mathcal A}_n=PDP^{-1}$. Therefore, for any $\alpha\in \mathbb{R}$ we can write ${\mathcal A}_n^\alpha=PD^\alpha P^{-1}$ or
	$$
	{\mathcal A}_n^\alpha=\begin{pmatrix}
		p_{11}     & p_{12}              & \cdots  & p_{1n} \\
		0     & p_{22}          & \cdots  & p_{2n} \\
		\vdots     & \vdots          & \ddots  & \vdots  \\
		0     & 0          & \cdots  & p_{nn} 
	\end{pmatrix}=\begin{pmatrix}
		b_{11}     & b_{12}               & \cdots  & b_{1n} \\
		0    & b_{22}          & \cdots  & b_{2n} \\
		\vdots     & \vdots          & \ddots  & \vdots  \\
		0     & 0          & \cdots  & b_{nn} 
	\end{pmatrix}\cdot\begin{pmatrix}
		a_{1}^\alpha   & 0     &         \cdots & 0 \\
		0       & a_{2}^\alpha &         \cdots & 0 \\
		\vdots  & \vdots& \ddots         & \vdots  \\
		0       & 0     &   \cdots       & a_{n}^\alpha 
	\end{pmatrix}\cdot \begin{pmatrix}
		c_{11} & c_{12}          & \cdots &   c_{1n} \\
		0 & c_{22}      & \cdots &   c_{2n} \\
		\vdots & \vdots      & \ddots &   \vdots  \\
		0 & 0      & \cdots &   c_{nn} 
	\end{pmatrix}=
	$$
	$$
	=\begin{pmatrix}
		b_{11}     & b_{12}               & \cdots  & b_{1n} \\
		0    & b_{22}          & \cdots  & b_{2n} \\
		\vdots     & \vdots          & \ddots  & \vdots  \\
		0     & 0          & \cdots  & b_{nn} 
	\end{pmatrix}\cdot
	\begin{pmatrix}
		a_{1}^\alpha c_{11} & a_{1}^\alpha c_{12}           & \cdots &   a_{1}^\alpha c_{1n} \\
		0 & a_{2}^\alpha c_{22}      & \cdots &   a_{2}^\alpha c_{2n} \\
		\vdots & \vdots      & \ddots &   \vdots  \\
		0 & 0      & \cdots &  a_{n}^\alpha  c_{nn} 
	\end{pmatrix}=
	$$
	$$
	=\begin{pmatrix}
		a_{1}^\alpha b_{11}c_{11}      & a_{1}^\alpha b_{11}c_{12} +	a_{2}^\alpha b_{12}c_{22}          & \cdots &   
		a_{1}^\alpha b_{11}c_{1n} +	a_{2}^\alpha b_{12}c_{2n}+...+a_{n}^\alpha b_{1n}c_{nn} \\
		0                      & a_{2}^\alpha b_{22} c_{22}      & \cdots &   a_{2}^\alpha b_{22}c_{2n} +	a_{3}^\alpha b_{23}c_{3n}+...+a_{n}^\alpha b_{2n}c_{nn} \\
		\vdots                                                                              & \vdots      & \ddots &   \vdots  \\
		0 & 0      & \cdots &  a_{n}^\alpha b_{nn} c_{nn} 
	\end{pmatrix}.
	$$
	Since $c_{mm}=\frac{1}{b_{mm}}=1$, $m=1,...,n$  on the main diagonal we obtain $p_{mm}=a_{m}^\alpha$, $m=1,...,n$. For $s<m$ we get
	$$
	p_{sm}=\sum\limits_{j=s}^m a_j^\alpha b_{sj}c_{jm}.
	$$
	For $b_{sj}$ and $c_{jm}$	we have representations by Theorem \ref{theo01}
	$$
	b_{sj}= \left\{ \begin{array}{ll}
		1 & \mbox{if $s=j$};\\
		0 & \mbox{if $j<s$};\\
		\frac{ \prod\limits_{i=s}^{j-1}a_i}{\prod\limits_{i=s}^{j-1}(a_i-a_j)} & \mbox{if  $s<j$},
	\end{array} \right. 
	\qquad
	c_{jm}= \left\{ \begin{array}{ll}
		1 & \mbox{if $j=m$;}\\
		0 & \mbox{if $m<j$};\\
		\frac{\prod\limits_{i=j}^{m-1}a_i}{\prod\limits_{i={j+1}}^{m}(a_{i}-a_{j})} & \mbox{if   $j< m$}.
	\end{array} \right.  
	$$
	
	When 
	$j=s<m$ we get 
	$$
	b_{ss}=1,\qquad c_{sm}=\frac{\prod\limits_{i=s}^{m-1}a_i}{\prod\limits_{i={s+1}}^{m}(a_{i}-a_{s})}=\frac{\prod\limits_{i=s}^{m-1}a_i}{\prod\limits_{i=s,i\neq s}^{m}(a_i-a_s) },
	$$
	when $j=m>s$ we get 
	$$c_{mm}=1,\qquad b_{sm}=\frac{ \prod\limits_{i=s}^{m-1}a_i}{\prod\limits_{i=s}^{m-1}(a_i-a_m)}=\frac{ \prod\limits_{i=s}^{m-1}a_i}{\prod\limits_{i=s,i\neq m}^{m}(a_i-a_m)}.$$
	
	For $s<j<m$, $s<m$	we have
	$$
	b_{sj}c_{jm}=\frac{ \prod\limits_{i=s}^{j-1}a_i}{\prod\limits_{i=s}^{j-1}(a_i-a_j)}\cdot \frac{\prod\limits_{i=j}^{m-1}a_i}{\prod\limits_{i={j+1}}^{m}(a_{i}-a_{j})}=\frac{\prod\limits_{i=s}^{m-1}a_i}{\prod\limits_{i=s,j\neq i}^{m}(a_i-a_j)},
	$$
	therefore for $s<m$
	$$
	p_{sm}=\sum\limits_{j=s}^m a_j^\alpha\frac{\prod\limits_{i=s}^{m-1}a_i}{\prod\limits_{i=s,j\neq i}^{m}(a_i-a_j)}=\prod\limits_{i=s}^{m-1}a_i\cdot\sum\limits_{j=s}^m\frac{ a_j^\alpha}{\prod\limits_{i=s,j\neq i}^{m}(a_i-a_j)}.
	$$	
	Taking into account the form of ${\mathcal A}_n^\alpha$  we get	that $p_{sm}$ can be written as \eqref{PowMatEl01}.
\end{proof}

\begin{remark}
	Taking into account Theorem \eqref{teopowder} we notice that  formula \eqref{PowMat03} for an arbitrary power $\alpha\in\mathbb{R}$ of two-band   matrix ${\mathcal A}_n$ is	matched with the particular case \eqref{PowMat} when $\alpha=k{\,\in\,}\mathbb{N}\cup\{0\}$ and formula \eqref{PowMat} was obtained by different way. So, we can introduce the generalisation of the complete homogeneous symmetric polynomial of degree $q$ in $m$ variables $a_1,...,a_m$ to the arbitrary power $\alpha\in\mathbb{R}$ by the formula
$$
		\mathscr{H}_{\alpha-m+1}(a_1,a_{2}...,a_m)=\sum\limits_{j=1}^m\frac{a_j^\alpha}{\prod\limits_{i=1,j\neq i}^{m}(a_j-a_i)}.
$$
\end{remark}

\section{Conclusion}\label{sec13}

Derivatives and integrals of a function with respect to another function are well-established concepts in basic calculus, relying on the chain rule and Riemann-Stieltjes integration, with natural applications in probability theory. In this article, we introduce the novel theory of fractional powers of such operators, developed via matrix approximations. First, the simplest but contentful example, which illustrates the naturalness of the matrix approach, is presented. From this example, it becomes clear that powers of the derivatives of a function with respect to another function are handy approximated by the powers of two-band matrices. 
Next, the positive integer powers of two-band matrices are obtained from the theory of complete homogeneous symmetric polynomials. Since we have the positive integer powers of operators, we can apply the semigroup theory and Balakrishnan's approach to construct fractional powers. 
In this way, we obtain general estimates for the rate of convergence of linear positive operators. 
To simplify future calculations, we find the arbitrary powers of two-band matrices.
The obtained results have many directions for further development which we are going to addressed to the next publications. First, obtaining fractional powers of operators with nonmonotonic measures with applications to theory of probability; second, solving fractional differential equations; third, numerical methods; fourth, obtaining fractional powers of gradients and other multidimensional differentiation operators; and much more.

\section*{Declarations}

\begin{itemize}
\item The authors declare that there is no conflict of interest.
\end{itemize}


\begin{thebibliography}{99}
\bibitem{SamkKilbMari1993} \textsc{Samko, S.G,    Kilbas, A.A.,  Marichev, O.I.}  \textit{Fractional integrals and derivatives}, Amsterdam, Gordon
and Breach Science Publishers,  {\bf 1993}.
\bibitem{Kolokoltsov2019}  \textsc{Kolokoltsov, V. N.} The Probabilistic Point of View on the Generalized Fractional Partial Differential Equations. \textit{Fract. Calc. Appl. Anal.}  {\bf 2019}, {\em 22}, 543--600.
\bibitem{Hilfer} \textsc{Hilfer, R.}  \textit{Anomalous Transport: Foundations and Applications}, Editors:   R. Klages,  G. Radons,  I. M. Sokolov, Eds.; Wiley-VCH Verlag GmbH \& Co. KGaA: Weinheim an der Bergstrasse, Germany, {\bf 2008}, 17--73. 
\bibitem{MariShish2024}	\textsc{ Marichev, O., Shishkina, E.} Overview of fractional calculus and its computer implementation in Wolfram Mathematica. \textit{Fract. Calc. Appl. Anal.}  {\bf 2024}, {\em 27}, 1995--2062.
\bibitem{Bochner} \textsc{ Bochner S.} Diffusion equation and stochastic processes. \textit{ Proc. Nat. Acad. Sci. U.S.A.} {\bf 1949}, {\em 35}, 368--370.
\bibitem{Phillips}	\textsc{Phillips, R.S.} On the generation of semigroups of linear operators. \textit{Pacific J. Math.} {\bf 1952}, {\em 2}, 343--369.
\bibitem{YosidaFP}  \textsc{Yosida, K.} Fractional powers of infinitesimal generators and the analyticity of the semigroups generated by them. \textit{Proc. Japan Acad.} {\bf 1960}, {\em 36},  86--89. 
\bibitem{Balakrishnan1959}  \textsc{ Balakrishnan, A.V.} An operational calculus for infinitesimal generators of semigroups. \textit{Trans. Amer. Math. Soc.} {\bf 1959}, {\em 91},  330--353. 
\bibitem{Westphal1970}  \textsc{ Westphal, U.} Ein Kalk\"{u}l f\"{u}r gebrochene Potenzen infinitesimaler Erzeuger
von Halbgruppen und Gruppen von Operatoren, Teil I: Halbgruppenerzeuger,
Teil II: Gruppenerzeuger'. \textit{ Gompositio Math.} {\bf 1970}, {\em 22}, 67--103, 104--136.
\bibitem{Yosida1980} \textsc{ Yosida, K.}  \textit{Functional Analysis}, 6 th ed., Springer-Verlag: Berlin, Germany, {\bf 1980}.
\bibitem{Alfonso} \textsc{ Jos\'{e} Alfonso L\'{o}pez Nicol\'{a}s}  \textit{ Derivative of a function with respect to a function: Differential Calculus of a function with respect to a function}, Independently published, {\bf 2019}.

\bibitem{Orsinger1} \textsc{Orsinger, E.,  Toaldo, B.} Space-time fractional equations and the related stable processes at random times. 
\textit{J. Theor. Probab.}, {\bf 2017}, {\em 30}, 1--26
\bibitem{Orsinger2} \textsc{ Garra, R., Orsingher, E., Polito, F.} A Note on Hadamard Fractional Differential Equations with Varying Coefficients and Their Applications in Probability. \textit{Mathematics}, {\bf 2018}, {\em 6}, No 4. 

\bibitem{Arran} \textsc{Fernandez, A., Fahad, H.M.}  A Historical Survey of Fractional Calculus with Respect to Functions ($\psi$-fractional Calculus). In: Ball, J., Tylli, HO., Virtanen, J.A. (eds) {\it Operator Theory, Related Fields, and Applications. IWOTA 2023. Operator Theory: Advances and Applications}, Birkh\"{a}user, Cham. {\bf 2025}, {\em 307}, 257--278.  


\bibitem{WiliamSwartz} \textsc{ Swartz, W.} Integration by matrix inversion. \textit{The American Mathematical Monthly}.
{\bf 1958}, {\em 65}, No 4, 282--283.
\bibitem{PodlubnyMatr01}  \textsc{ Podlubny, I.} Matrix approach to discrete fractional calculus. \textit{Fractional Calculus and Applied Analysis}. {\bf 2000}, {\em 3}, No 4, 359--386.
\bibitem{PodlubnyMatr02}  \textsc{ Podlubny, I., Chechkin, A.,  Skovranek, T.,  Chen, Y.-Q.,  Vinagre Jara, B.M.}
 Matrix approach to discrete fractional calculus II: partial fractional differential equations. \textit{Journal of Computational Physics}.
 {\bf 2009}, {\em 228}, No 8,  3137--3153.
 \bibitem{PodlubnyMatr03}
\textsc{ Podlubny, I.,  Skovranek, T.,  Vinagre Jara, B. M.} Matrix approach to discretization of fractional derivatives and to solution of fractional differential equations and their systems. \textit{IEEE Conference on Emerging Technologies \& Factory Automation}, Palma de Mallorca, Spain. {\bf 2009}, 1--6.
\bibitem{Podlubny}  \textsc{ Podlubny, I.}
\emph{Fractional differential equations. An introduction to fractional derivatives, fractional differential equations, to methods of their solution and some of their applications}. Mathematics in Science and Engineering, 198, Academic Press, SanDiego, {\bf 1999}.
\bibitem{KolokoltsovBook}
\textsc{ Kolokoltsov, V.N.} \textit{Markov Processes, Semigroups and Generators}, Berlin, New York: De Gruyter, {\bf 2011}. 
\bibitem{KolokoltsovShishkina}
\textsc{ Kolokoltsov, V.N., Shishkina, E.L.} Fractional Calculus for
Non-Discrete Signed Measures. \textit{Mathematics}, {\bf 2024}, {\em 12}, 2804.
\bibitem{Khetchatturat} \textsc{ Khetchatturat, C., Leerawat, U.,  Siricharuanun, P.}  Powers of Some Special Upper Triangular Matrices. \textit{Thai Journal of Mathematics}, {\bf 2024}, {\em 22}, No 1, 111--118.
\bibitem{Cornelius} \textsc{ Cornelius Jr., E.F.} Identities for complete homogeneous symmetric polynomials.  
\textit{Journal of Algebra, Number Theory and Applications}, {\bf 2011}, {\em 21}, No 1, 109--116.
\bibitem{Bhatnagar} \textsc{Bhatnagar, G.} A short proof of an identity of Sylvester, Internat. J. Math. \& Math. Sci.
22.2 (1999), pp. 431--435.
\bibitem{Exton} \textsc{ Exton, H.}  \textit{ $q$-Hypergeometric Functions and Applications}, New York, Halstead Press, Chichester, Ellis Horwood, {\bf 1983}.

\end{thebibliography}
\end{document}